\documentclass[oneside, reqno, 11pt, a4paper]{amsart}

\AtBeginDocument{
  \DeclareSymbolFont{AMSb}{U}{msb}{m}{n}
  \DeclareSymbolFontAlphabet{\mathbb}{AMSb}
}
\DeclareFontFamily{U}{mathx}{\hyphenchar\font45}
\DeclareFontShape{U}{mathx}{m}{n}{<-> mathx10}{}
\DeclareSymbolFont{mathx}{U}{mathx}{m}{n}
\DeclareMathAccent{\widebar}{0}{mathx}{"73}
 
\usepackage{aligned-overset}
\usepackage{thm-restate}

\usepackage[dvipsnames]{xcolor}
\usepackage{graphicx}
\usepackage{svg}
\usepackage{tikz}
\usepackage{subcaption}
\tikzstyle{arrow} = [thick,->,>=stealth]
\usetikzlibrary{shapes.misc,matrix,fit,positioning,arrows.meta,decorations.pathreplacing,calc,shapes.geometric,arrows}
\tikzset{Matrix/.style={matrix of nodes, font=\footnotesize,text height=1pt, text depth=0.5pt, text width=8.5pt, align=center, column sep=0pt, row sep=0pt, nodes in empty cells}}

\graphicspath{{figures/}}

\usepackage[a4paper, pdftex, left=2cm, top=2cm, right=2cm, bottom=2cm]{geometry}
\usepackage[final]{pdfpages}
\usepackage{changepage}

\usepackage[foot]{amsaddr}
\numberwithin{equation}{section}

\usepackage{url}
\usepackage{hyperref}
\hypersetup{
    breaklinks=true,
    bookmarksopen=true,
    pdftitle={BDDC solvers for HHO}, 
    pdfauthor={Santiago Badia, Jerome Droniou, Jai Tushar}, 
	pdfauthor={}
    pdfsubject={}, 
    colorlinks=true,
    linkcolor=black,
    citecolor=blue,
    filecolor=black,
    urlcolor=blue
}

\usepackage{csquotes}
\usepackage[backend=biber, defernumbers=true, maxbibnames=5, style=numeric-comp, isbn=false, bibencoding=utf8, safeinputenc, url=false, doi=true, giveninits=true]{biblatex}

\usepackage{mathtools}
\usepackage{mathabx}

\usepackage{float}
\usepackage{framed}
\usepackage{verbatim}
\usepackage{fancyvrb}
\usepackage{booktabs}
\usepackage{colortbl}
\usepackage{multirow}
\usepackage{algorithm}
\usepackage{algpseudocode}
\usepackage{cancel}
\usepackage{accents}
\usepackage{mleftright}

\makeatletter
\algdef{S}[IF]{IfNoThen}[1]{\algorithmicif\ #1}
\makeatother
\usepackage[inline]{enumitem}
\usepackage{siunitx}

\newtheorem{theorem}{Theorem}
\numberwithin{theorem}{section}
\newtheorem{remark}[theorem]{Remark}
\newtheorem{lemma}[theorem]{Lemma}

\newtheorem{proposition}[theorem]{Proposition}
\newtheorem{assumption}[theorem]{Assumption}
\usepackage[breakable]{tcolorbox}
\tcbuselibrary{theorems}
\tcbuselibrary{skins}
\tcbset{
	commonstyle/.style={
		theorem style=plain,
		enhanced jigsaw,
		fonttitle=\bfseries,
		fontupper=\itshape,
		halign=justify,
		separator sign=:,
		description delimiters none,
		description font=\bfseries, 
		terminator sign={.\hspace{0.25em}},
		arc=0mm,outer arc=0mm,
		boxrule=0pt,toprule=0pt,bottomrule=0pt,leftrule=0pt,rightrule=0pt,
		titlerule=0pt,toptitle=0pt,bottomtitle=0pt,top=0pt,
		colback=white,coltitle=black,
		boxsep=0pt, bottom=0pt, left=0pt, 
	}
}
\newtcbtheorem[]{myproblem}{Problem}%
{center, commonstyle, fonttitle=\bfseries}{pb}
\newtcbtheorem[]{mydefinition}{Definition}
{center, commonstyle, fonttitle=\bfseries}{pb}
\newtcbtheorem[]{myassumption}{Assumption}
{center, commonstyle, fonttitle=\bfseries}{pb}

\usepackage{amsmath, amsfonts, amssymb, amscd, bm, mathtools}
\newcommand{\Mh}{\mathcal{M}_{h}}
\newcommand{\Th}{\mathcal{T}_{h}}
\newcommand{\Thloc}{\mathcal{T}^i_{h}}
\newcommand{\Fh}{\mathcal{F}_{h}}

\newcommand{\interior}{{\rm in}}
\newcommand{\boundary}{{\rm bd}}
\newcommand{\Fhi}{\Fh^{\interior}}
\newcommand{\Fhb}{\Fh^{\boundary}}
\newcommand{\Fhbloc}[1]{\Fh^{\boundary,#1}}
\newcommand{\FFhb}{\mathcal{F\!F}_h^{\boundary}}
\newcommand{\FFhbloc}{\mathcal{F\!F}_h^{\boundary,i}}

\newcommand{\lproj}[2]{\pi_{#1}^{#2}}
\newcommand{\projFace}{\mathfrak{F}}
\newcommand{\ul}[1]{\underline{#1}}
\newcommand{\Uh}{\ul{U}_{h}}

\newcommand{\Uhi}{U_{h}}

\newcommand{\Uhloc}{\ul{U}^i_{h}}

\newcommand{\Uhb}{U_{h}^{\partial}}

\newcommand{\Uhbb}{U_{h}^{\partial,\boundary}}

\newcommand{\Uhbbloc}{U_{h}^{\partial,\boundary,i}}

\newcommand{\tr}{\gamma}

\newcommand{\ffp}{{f \! f' }}

\newcommand{\dffp}{|x_f-x_{f'}|}

\newcommand{\dist}{\mathop{\rm dist}}
\newcommand{\distV}{\mathop{\rm dist}\nolimits_V}
\newcommand{\distH}{\mathop{\rm dist}\nolimits_H}

\newcommand{\norm}[2]{\|#2\|_{#1}}
\newcommand{\Norm}[2]{\left\|#2\right\|_{#1}}
\newcommand{\seminorm}[2]{|#2|_{#1}}
\newcommand{\tnorm}[2]{\vert #2\vert_{#1}}
\newcommand{\Poly}[1]{\mathbb{P}_{#1}}
 
\newcommand{\dom}{\Omega}

\newcommand{\An}[3]{{\ocirc}_{#2}^{#1}(#3)}
\newcommand{\Above}{\mathcal{V}}

\newcommand{\ol}[1]{\overline{#1}}

\newcommand{\A}{A}

\newcommand{\La}{\mathcal{L}}

\newcommand{\Ca}{\mathcal{C}_{\ffp s}}

\newcommand{\W}{\mathcal{W}}

\newcommand{\Iffp}{\mathcal{I}_{\ffp}}

\newcommand{\fint}{g}
\newcommand{\fbou}{f}

\newcommand{\HoneOp}{\boldsymbol{\mathcal{A}}}
\newcommand{\HhalfOp}{\boldsymbol{\mathcal{H}}}

\usepackage{acronym}
\usepackage{enumitem}
\graphicspath{{figures/}}

\title[]{A discrete trace theory for non-conforming polytopal hybrid discretisation methods
} 
\date{\today}

\address{$^\dagger$School of mathematics\\Monash university\\Clayton\\Victoria 3800\\Australia}
\address{$^\sharp$IMAG, Univ. Montpellier, CNRS, Montpellier, France.}
\address{$^*$Corresponding author}
\author[S. Badia]{Santiago Badia$^{\dagger}$}
\email{santiago.badia@monash.edu}
\author[J. Droniou]{Jerome Droniou$^{\sharp\dagger*}$}
\email{jerome.droniou@umontpellier.fr}
\author[J. Tushar]{Jai Tushar$^{\dagger}$}
\email{jai.tushar@monash.edu}

\begin{document}

\begin{abstract}
	In this work we develop a discrete trace theory that covers non-conforming hybrid discretization methods and holds on polytopal meshes. A notion of discrete trace seminorm is defined, and {\it trace} and {\it lifting} results with respect to a discrete $H^1$-seminorm on the hybrid fully discrete space are proven. Finally, we conduct a numerical test in which we compute the proposed discrete operators and investigate their spectrum to verify the theoretical analysis. The development of this theory is motivated by the design and analysis of preconditioners for hybrid methods, e.g., of substructuring domain decomposition type.
\end{abstract}

\maketitle

\tableofcontents

\section{Introduction}
\label{sec:introduction}

Trace theory is at the heart of condition number analysis of several sub-structuring non-overlapping Domain Decomposition Methods (DDMs), such as Balancing Domain Decomposition by Constraints (BDDC) and Finite Element Tearing and Interconnecting -- Dual Primal (FETI-DP) \cite{Dohrmann-BDDC-Algorithm,Toselli-Widlund-DDM-book}. In the analysis of DDMs, trace theory is required to bound the injection operator from the space of discontinuous solutions (across inter-subdomain boundaries) into the original space, defined via a weighting operator and a subdomain-local discrete harmonic extension. 
The continuity of this injection relies on an equivalence between the seminorm induced by the bilinear form on the subdomain $\Omega$ and the $H^{1/2}(\partial \dom)$-seminorm (the trace seminorm) of truncations of functions to subdomain faces. 

When using grad-conforming Finite Element Methods (FEMs) for the discretisation, one can rely on the continuous trace theory and $H^{1/2}(\partial \Omega)$-seminorm, which is characterized by the Sobolev--Slobodeckij seminorm \cite{Nezza-FractionalSobolevSpace}
\begin{align}\label{cts:Honehalf}
  |v|_{1/2, \partial \dom}^2 \doteq \int_{\partial \dom} \int_{\partial \dom} \frac{|v(x) - v(y)|^2}{|x - y|^{d}}.
\end{align}
This seminorm satisfies two well-known properties, namely, \emph{trace inequality}  and \emph{lifting} \cite[Theorem 4.1]{Xu-Zou-1998-NonoverlappingDD}. In DDM analysis, the former implies that the restriction of functions to the subdomain interface is stable and the latter then lifts this restriction to the interior of the neighbouring subdomain.

For non-conforming FEMs, like mixed FEM~\cite{Boffi-Brezzi-Fortin-Mixed-FEM-Book}, Discontinuous Galerkin (DG) methods~\cite{Arnold-Brezzi-Cockburn-Marini-Unified-DG}, Hybridizable DG (HDG)~\cite{Cockburn-Gopalakrishnan-Lazarov-HDG}, non-conforming Virtual Elements (ncVEM)~\cite{Dios-Lipnikov-Manzini-ncVEM}, Weak Galerkin (WG)~\cite{Wang-Ye-WG} and Hybrid High Order (HHO)~\cite{di-pietro.ern:2015:hybrid}, showing the seminorm equivalence is not trivial. This is because the trace of piecewise polynomial functions in $L^2(\dom)$ do not have $H^{1/2}(\partial \dom)$-regularity. A first breakthrough in this direction was made for mixed FEMs and Balancing Neumann-Neumann (BNN) preconditioning in \cite{Cowsar-Mandel-Wheeler-1995-BDDC-Mixed}. There, the authors addressed this issue by constructing an interpolant of a function on the interface onto a conforming FE space, and then showed equivalence of its trace seminorm with the norm induced by the bilinear form on the subdomain interface. The same idea has been applied to a large class of DG methods and BDDC preconditioning in \cite{Diosady-Darmofal-David-2012-BDDC-DG}. Furthermore, the authors in \cite{Tu-Wang-BDDC-2016-HDG} combine this technique with a spectral equivalence established between HDG and hybridized RT methods in \cite{Cockburn-Dubios-Gopalakrishnan-Tan-2014-Multigrid-HDG} to show quasi-optimal condition number bounds for the BDDC preconditioning of HDG methods. The same strategy as in HDG is used for WG methods in \cite{Tu-Wang-2018-BDDC-WG}. However, since this approach relies on conforming interpolants, it cannot be extended to polytopal meshes, one of the main motivation behind the use of non-conforming methods. All the analyses for non-conforming spaces so far have been carried out on conforming simplicial or quadrilateral/hexahedral meshes. Besides, to the best of our knowledge, there is no literature available addressing ncVEM or HHO.

As discussed above, the main challenge when designing and analysing DDMs for non-conforming methods on polytopal meshes is the lack of a complete trace theory. This is the main motivation behind this work. We design a discrete trace seminorm, which avoids the need for conforming interpolants, and prove that it enjoys analogous properties (trace and lifting) with respect to a discrete $H^1(\dom)$-seminorm as the conforming counterparts. As far as we know, this path has never been explored before.

The closest work we could identify in this direction is \cite{Kashiwbara-Takahito-Issei-Zhou-2019-DiscreteHonehalf}, in which the authors design such a norm for the non-conforming Crouzeix--Raviart finite element and combine it with an enrichment process to prove a discrete lifting result. However, this norm is too strong to also satisfy the trace inequality, mainly because the trace seminorm is not easily localisable. To see this, letting $\{\tau\}$ be a family of simply connected subdomains that partition $\partial \dom$, we can rewrite \eqref{cts:Honehalf} as follows
\begin{align*}
  | v |_{1/2, \partial \dom}^2 = \sum_{\tau} | v |_{1/2, \tau}^2 + \sum_{\tau} \int_{\tau} \int_{\partial \dom \backslash \tau} \frac{|v(x) - v(y)|^2}{|x - y|^{d}}.
\end{align*}
The second term in the above equality is difficult to localise. Some studies tackling this aspect can be found in  \cite{Faermann-2002-Localization-BEM,bertoluzza2023localization} and references therein. Our design of the discrete $H^{1/2}(\partial \dom)$-seminorm (see Section \ref{sec:DiscHonehalf}) mimics the above decomposition in the discrete setting, using the mesh size as a scaling factor. 

Conceptually, our approach is close to the techniques of Discrete Functional Analysis (DFA) as developed, e.g., in \cite{eymard.gallouet.ea:2010:discretization,Droniou.Eymard:2018:GDM}. These techniques have however, so far, only been used to prove inequalities involving as boundary norm the $L^2(\partial\dom)$-norm (which prevented from proving a lifting property), and never to design or analyse a discrete $H^{1/2}(\partial\dom)$-seminorm.
Following DFA principles, we only manipulate here discrete spaces, whose elements are vectors of polynomials in the mesh cells and on the mesh faces. Our arguments do not invoke any continuous trace or lifting result; instead, we mimic at the discrete level the proofs of these continuous results. This mimicking requires to construct various specific sets of cells or faces, on which estimates need to be established.

We finally note that, even though our driving motivation for a discrete trace theory is the design of robust and scalable preconditioners for non-conforming polytopal discretisations, this theory is of interest in itself, and can be applied in other scenarios -- e.g., in the design and analysis of operator preconditioners  involving trace inner products \cite{Kuchta2016}. 

\subsection{Main contributions and outline}
The main contribution of this work is the development of a discrete trace theory on polytopal meshes spanning a large class of hybrid methods. It hinges on the design and analysis of a discrete trace seminorm. We define it in Section \ref{sec:DiscHonehalf} and state its two main properties, namely, discrete trace inequality in Theorem \ref{thm:trace} and lifting property in Theorem \ref{thm:lifting}. These results are dimension-independent and proved under a quasi-uniformity assumption on the polytopal mesh (see Assumption \ref{assum:reg.mesh}). 

The proofs of these theorems are structured into two main stages. In the first stage, we establish the discrete trace inequality and lifting in the flat case, when $\dom$ is a cube and we only consider one side of its boundary. This stage is detailed in Sections \ref{sec:proof.trace.cube} and \ref{sec:proof.lifting.cube}, respectively. It mimics the proofs done in the continuous settings (briefly recalled in Sections \ref{sec:trace.continuous} and \ref{sec:lifting.continuous}, respectively) and additionally relies on some results on partitionings of the mesh cells and faces, whose proofs are postponed at the end of their corresponding sections for better readability (Sections \ref{sec:proof.technical.trace} and \ref{sec:proof.technical.lifting}, respectively). In the second stage, we extend these results to a generic polytopal domain $\dom$ by localising and gluing along $\partial \dom$. This stage is detailed in Section \ref{sec:extensions}, in which we also demonstrate that the trace inequality and lifting results also apply to all kinds of discrete spaces encountered in polytopal methods.

In Section \ref{sec:NumExp} we numerically illustrate the theoretical results. An appendix, Section \ref{sec:appendix}, provides two tables gathering the main notations used in the paper; the reader is encouraged to use these tables as they go over the most technical aspects of the proofs.

% The proofs of the trace inequality and lifting are in separate subsections (see Sections \ref{sec:proof.trace} and \ref{sec:proof.lifting}). The analysis is structured into two main stages. In the first stage, we prove these results in the flat case, when $\Omega$ is a cube and we only consider one side of it as boundary. It relies on some preliminary results on partitioning of this half space, which are detailed in Section \ref{sec:prelim}. In the second stage (see Section \ref{sec:polytopalcase}) we extend these results to a generic polytopal $\dom$ by localising and gluing along $\partial \dom$.

\section{Discrete $H^{1/2}(\partial \dom)$-seminorm, trace and lifting}\label{sec:DiscHonehalf}
Let $\dom \subset \mathbb{R}^d$ $(d \geq 2)$ be a polytopal domain with boundary $\partial \dom$. We consider a partition $\Th$ of $\dom$ into a finite collection of non-empty disjoint polytopes, called \emph{cells}, and define its mesh \emph{skeleton} as $\bigcup_{t \in \Th} \partial t$. For each cell $t \in \Th$, we denote by $\mathcal{F}_t$ the set of faces of $t$, which are $(d-1)$-dimensional polytopes.
%, and define the set of neighbouring cells by
%\[
%\N_t \doteq \{ t'\in\Th :|t \cap t'|_{d-1} \neq 0 \}.
%\]
The set of boundary faces, that is, the partition of the domain boundary $\partial \Omega$ generated by the domain partition, is denoted by 
\[
\Fhb \doteq \bigcup_{t \in \mathcal{T}_h} \{ f \in \mathcal{F}_t \ : \ f \subset \partial \Omega \},
\] 
and the set of interior faces by
\[
\Fhi \doteq \bigcup_{t\in\mathcal T_h}\{f\in\mathcal F_t\ :\ f\subset\Omega\}.
\] 
We denote by $\Fh \doteq \Fhb \cup \Fhi$ the set of all faces of the mesh.

We combine the domain and skeleton meshes into a \emph{hybrid} mesh $\Mh \doteq (\Th, \Fh)$ (see \cite[Definition 1.4]{di-pietro.droniou:2020:hybrid}). The diameter of $t \in \Th$ is denoted with $h_t$, and we define the \emph{mesh size} $h = \max_{t \in \Th} h_t$. Analogously, we denote by $h_f$ the diameter of $f \in \Fh$. 

Throughout the paper, $|X|_{n}$ is the $n$-dimensional measure of a set $X\subset \mathbb{R}^d$. When $X$ is a set of cells/faces, with an abuse of notation we write $|X|_n$ for the $n$-dimensional measure of the union of the cells/faces in $X$.

We assume that the hybrid mesh of $\dom$ is regular and quasi-uniform.
Some of the arguments developed in the proofs can easily be adapted to non-quasi-uniform meshes; others are, however, much less obvious to establish when the mesh is not quasi-uniform. Considering that many of our proofs are already quite technical, for legibility purposes we decided to consider quasi-uniform meshes throughout the paper.

\begin{assumption}[Mesh regularity]\label{assum:reg.mesh}
  The mesh $\Mh$ is regular as per \cite[Definition 1.9]{di-pietro.droniou:2020:hybrid}, and quasi-uniform in the sense that there exists $\varrho>0$ independent of $h$ such that, for all $t\in\Th$, $h\le \varrho h_t$. 
\end{assumption} 

On the hybrid mesh, we define bulk and trace piecewise polynomial functional spaces as follows. 
We consider the \emph{{hybrid} space} $\Uh = \Uhi \times \Uhb$, where 
\[
\Uhi \doteq \bigtimes_{t \in \Th} \mathbb{P}_k(t), \quad
\Uhb \doteq \bigtimes_{f \in \Fh} \mathbb{P}_k(f),  
\] 
and $\mathbb{P}_k$ is the space of polynomials of total degree $k$ or less.

We define the following discrete $H^1(\dom)$-seminorm on $\Uh$, inspired by the HHO theory \cite[Section 2.2.1]{di-pietro.droniou:2020:hybrid} (albeit with a slightly different scaling of boundary terms, see \cite[Eq.~(2.4)]{droniou.yemm:2021:robust}): For all $\ul{v}_h = ((v_t)_{t \in \Th}, (v_f)_{f \in \mathcal{F}_t}) \in \Uh$, 
\begin{equation}\label{eq:def.disc.H1}
  \seminorm{1,h}{\ul{v}_h}\doteq \left( \sum_{t\in\Th}\seminorm{1,t}{\ul{v}_t}^2\right)^{1/2}\quad\mbox{ with }\quad
  \seminorm{1,t}{\ul{v}_t}^2\doteq \norm{L^2(t)}{\nabla v_t}^2+\sum_{f\in\mathcal{F}_t}h_t^{-1}\norm{L^2(f)}{v_f-v_t}^2.
\end{equation}
 We also design here a novel notion of discrete $H^{1/2}(\partial \dom)$-seminorm on restrictions of hybrid spaces to the boundary. Setting 
 $$
  \Uhbb \doteq \{ w_h = (w_f)_{f \in \Fhb} : w_f \in \mathbb{P}_k(f) \;\; \forall f \in \Fhb \},
 $$
 we have the natural trace operator $\tr:\Uh \to \Uhbb$ such that
   \[
   \tr (\ul{v}_h)|_f = v_f\qquad\forall f\in\Fhb\,,\quad\forall \ul{v}_h\in\Uh.
   \]
The discrete $H^{1/2}(\partial \dom)$-seminorm of $w_h\in\Uhbb$ is defined by
\begin{equation}\label{eq:def.disc.Hhalf}
  \tnorm{1/2,h}{w_h}^2 \doteq \sum_{f\in\Fhb} h_f^{-1} \norm{L^2(f)}{w_f - \ol{w}_f}^2 + \sum_{(f,f')\in\FFhb} |f|_{d-1} |f'|_{d-1} \frac{|\ol{w}_f - \ol{w}_{f'}|^2}{\dffp^d},
\end{equation}
where $\ol{w}_f \doteq \frac{1}{|f|_{d-1}} \int_{f} w_f $ is the average of $w_f$ over $f$,
\begin{equation}\label{eq:def.FFhb}
  \FFhb=\{(f,f')\in\Fhb\times\Fhb\,:\,f\neq f'\},
\end{equation}
and $x_f$ and $x_{f'}$ are the centroids\footnote{``The'' centroid $x_f$ of a face $f$ is actually arbitrarily chosen such that $f$ contains a ball centered at $x_f$ and of radius $\varpi h_f$, where $\varpi$ only depends on the mesh regularity parameter in \cite[Definition 1.9]{di-pietro.droniou:2020:hybrid}; a similar definition of centroid is adopted for cells.} of $f$ and $f'$. In the right-hand side of \eqref{eq:def.disc.Hhalf}, the first term measures the local variations\footnote{Note that, by a discrete inequality, $h_f^{-1} \norm{L^2(f)}{w_f - \ol{w}_f}^2\simeq \seminorm{H^{1/2}(f)}{w_f}^2$. However, the form we have chosen is more suitable for the subsequent analysis.}, while the second term measures the long-range variations. In the continuous setting, both terms are embedded into the double integral  \eqref{cts:Honehalf} that defines the $H^{1/2}(\partial \dom)$-seminorm. Separating the two contributions is natural in the discrete setting, due to the introduction of a particular scale (the mesh size). 

From here onwards, we write $a\lesssim b$ to indicate that $a\le Cb$ with $C\ge 0$ depending only on $\dom$, the mesh regularity parameter in \cite[Definition 1.9]{di-pietro.droniou:2020:hybrid}, the quasi-uniformity parameter $\varrho$ and, possibly, on the polynomial degrees involved in $a,b$. The notation $a\simeq b$ means that $a \lesssim b \lesssim a$.

The two main properties of the discrete $H^{1/2}(\partial \dom)$-seminorm are summarised in the following theorems, which form the main contributions of this work. 

\begin{theorem}[Trace inequality]\label{thm:trace}
  The following discrete trace inequality holds:
  \begin{equation}\label{eq:trace}
    \tnorm{1/2,h}{\tr (\ul{v}_h)} \lesssim \seminorm{1,h}{\ul{v}_h}\qquad\forall \ul{v}_h\in\Uh.
  \end{equation}
\end{theorem}

\begin{theorem}[Lifting]\label{thm:lifting}
  There exists a linear lifting operator $\mathcal{L}_h: \Uhbb \rightarrow \Uh$ such that: 
  \begin{equation}\label{eq:lifting}
    \seminorm{1,h}{\mathcal{L}_h(w_{h})}\lesssim \tnorm{1/2, h}{w_{h}}\quad\forall w_h\in\Uhbb, \qquad \gamma \circ \mathcal{L}_h  = \mathrm{Id}_{\Uhbb}.
  \end{equation}
\end{theorem}

\begin{remark}[Scaling]\label{scaling.dom}
  It can be checked that, if $\dom$ is scaled by a factor $\mu$, both the discrete $H^{1/2}(\partial \dom)$ and $H^{1}(\dom)$-seminorms scale with the factor $\mu^{d-2}$. Thus, the hidden constants in the trace inequality and lifting property stated above do not depend on $\mathrm{diam}(\dom)$ under rescaling, and we can assume that $\mathrm{diam}(\dom)$ is of unit size.
\end{remark}

\begin{remark}[Extension to more general hybrid spaces]
We show in Section \ref{sec:general.hybrid.space} that Theorems \ref{thm:trace} and \ref{thm:lifting} directly provide a trace inequality and a lifting property in more general hybrid spaces, which cover a wide range of polytopal methods.
\end{remark}

\begin{remark}[Discrete trace theory in non-Hilbertian spaces]
We develop in this paper the discrete trace theory in the Hilbertian setting, but the proofs are only based on Cauchy--Schwarz inequalities and counting or measuring various sets of faces and cells. As a consequence, they can be easily adapted to the $W^{1,p}$ setting with $p\in (1,\infty]$.
Theorems \ref{thm:trace} and \ref{thm:lifting} then hold simply by replacing the discrete $H^1(\Omega)$-seminorm \eqref{eq:def.disc.H1} by the discrete $W^{1,p}(\Omega)$-seminorm
\[
 \seminorm{1,p,h}{\ul{v}_h}\doteq \left( \sum_{t\in\Th}\norm{L^p(t)}{\nabla v_t}^p+\sum_{f\in\mathcal{F}_t}h_t^{1-p}\norm{L^p(f)}{v_f-v_t}^p\right)^{1/p}
 \]
(see \cite[Eq.~(6.9)]{di-pietro.droniou:2020:hybrid}) and the discrete $H^{1/2}(\partial\Omega)$-seminorm \eqref{eq:def.disc.Hhalf} by the discrete $W^{1-\frac1p,p}(\partial\Omega)$-seminorm 
\[
 \tnorm{1-\frac1p,h}{w_h} \doteq \left(\sum_{f\in\Fhb} h_f^{1-p} \norm{L^p(f)}{w_f - \ol{w}_f}^p + \sum_{(f,f')\in\FFhb} |f|_{d-1} |f'|_{d-1} \frac{|\ol{w}_f - \ol{w}_{f'}|^p}{\dffp^{d+p-2}}\right)^{1/p}
\]
(both seminorms are written here for $p<\infty$, the changes for $p=\infty$ being classical).
\end{remark}

\section{Discrete trace inequality on the side of a cube}\label{sec:proof.trace.cube}

The proof of Theorem \ref{thm:trace} consists in adapting to the discrete setting the arguments in \cite{Droniou_intsob:01} for the continuous trace inequality. In Section \ref{sec:trace.continuous}, we briefly recall these arguments in the case where $\Omega$ is a half-space, to make the technical proof in the discrete setting easier to follow. In Section \ref{sec:proof.trace}, we establish the discrete trace inequality when $\Omega$ is a cube and the trace is restricted to one of its sides. The case of a polytopal domain and trace on the complete boundary $\partial\dom$ is handled in Section \ref{sec:trace.lift.polytopalcase}, by combining the case of a cube and localisation arguments.

\subsection{Continuous trace inequality}\label{sec:trace.continuous}

We consider here that $\Omega=\mathbb{R}^{d-1}\times (0,\infty)$ is a half-space, we take $u$ a continuously differentiable function with compact support in $\mathbb{R}^{d-1}\times [0,\infty)$, and we establish the following trace inequality on the side of $\Omega$:
\begin{equation}\label{trace:continuous}
  \mathcal S\doteq\int_{\mathbb{R}^{d-1}}\int_{\mathbb{R}^{d-1}}\frac{|u(x',0)-u(y',0)|^2}{|x'-y'|^{d}}\,dx'dy'\lesssim \int_\Omega |\nabla u|^2.
\end{equation}

The idea is to estimate $u(x',0)-u(y',0)$ by going ``inside'' $\Omega$, at a distance $\ell\doteq |x'-y'|$ from $(x',0)$ and $(y',0)$, and by writing this
difference as a sum of three differences:
\begin{equation}
u(x',0)-u(y',0)=(u(x',0)-u(x',\ell))+(u(x',\ell)-u(y',\ell))+(u(y',\ell)-u(y',0)).
\label{continuous:decomp.u.minus.u}
\end{equation}
Taking the absolute value, integrating over $x',y'$ and using triangle inequalities, we deduce $\mathcal S\lesssim \mathcal S_1+\mathcal S_2+\mathcal S_3$ with
\begin{equation}\label{continuous:S1S2S3}
\begin{aligned}
\mathcal S_1\doteq &\int_{\mathbb{R}^{d-1}}\int_{\mathbb{R}^{d-1}}\frac{|u(x',0)-u(x',\ell)|^2}{|x'-y'|^{d}}\,dx'dy'\\
\mathcal S_2\doteq &\int_{\mathbb{R}^{d-1}}\int_{\mathbb{R}^{d-1}}\frac{|u(x',\ell)-u(y',\ell)|^2}{|x'-y'|^{d}}\,dx'dy'\\
\mathcal S_3\doteq &\int_{\mathbb{R}^{d-1}}\int_{\mathbb{R}^{d-1}}\frac{|u(y',\ell)-u(y',0)|^2}{|x'-y'|^{d}}\,dx'dy'=\mathcal S_1,
\end{aligned}
\end{equation}
the last equality coming from the symmetric roles of $x',y'$ in these integrals. We therefore only have to prove that $\mathcal S_1$
and $\mathcal S_2$ are bounded by the right-hand side of \eqref{trace:continuous} to conclude.

\medskip

\textit{\ul{Bound on $\mathcal{S}_1$}}: We start with a polar change of variable $y'=x'+\ell \xi\to (\ell,\xi)$ (see \cite[Chapter 6, (9)]{Stein-Elias-2005-Book}) to get
\begin{equation}\label{eq:S1.polar}
\mathcal S_1= \int_0^\infty\int_{\mathbb S^{d-2}}\int_{\mathbb{R}^{d-1}}\frac{|u(x',0)-u(x',\ell)|^2}{\ell^2}\,dx' d\xi d\ell,
\end{equation}
where $\mathbb S^{d-2}$ is the unit sphere in $\mathbb R^{d-1}$. We then substitute 
\begin{equation}\label{eq:diff.vert}
|u(x',0)-u(x',\ell)|\le \int_0^\ell |\partial_d u(x',s)|\,ds
\end{equation}
to obtain
\begin{align}
\mathcal S_1\le{}& \int_0^\infty|\mathbb S^{d-2}|_{d-2}\int_{\mathbb{R}^{d-1}}\left|\frac{1}{\ell}\int_0^\ell |\partial_d u(x',s)|\,ds\right|^2\,dx' d\ell\nonumber\\
={}&\int_0^\infty|\mathbb S^{d-2}|_{d-2}\Norm{L^2(\mathbb R^{d-1})}{\frac{1}{\ell}\int_0^\ell |\partial_d u(\cdot,s)|\,ds}^2.
\label{eq:S1.polar.2}
\end{align}
Invoking the classical formula 
\begin{equation}\label{eq:norm.integral}
\Norm{L^2(\mathbb{R}^{d-1})}{\int_0^\ell f(\cdot,s)\,ds}\le \int_0^\ell \norm{L^2(\mathbb{R}^{d-1})}{f(\cdot,s)}\,ds,
\end{equation}
we infer
\begin{equation}\label{eq:S1.polar.3}
\mathcal S_1\le |\mathbb S^{d-2}|_{d-2}\int_0^\infty\left(\frac{1}{\ell}\int_0^\ell\norm{L^2(\mathbb R^{d-1})}{\partial_d u(\cdot,s)}\,ds\right)^2.
\end{equation}
Let $F_1(\ell)= \frac{1}{\ell} \int_0^\ell f_1(s)\,ds$ with $f_1(s)=\norm{L^2(\mathbb{R}^{d-1})}{\partial_d u(\cdot,s)}$.
The integral Hardy inequality gives 
\[
\norm{L^2(0,\infty)}{F_1}\lesssim \norm{L^2(0,\infty)}{f_1}=\norm{L^2(\Omega)}{\partial_d u},
\]
which yields the desired estimate on $\mathcal S_1$:
\begin{equation*}%\label{eq:continuous.S1.final}
\mathcal S_1\le  |\mathbb S^{d-2}|_{d-2}\norm{L^2(0,\infty)}{F_1}^2\lesssim \norm{L^2(\Omega)}{\partial_d u}^2.
\end{equation*}

\medskip

\textit{\ul{Bound on $\mathcal{S}_2$}}: Recalling that $y'=x'+\ell\xi$ with $\xi$ unit vector, we use the Cauchy--Schwarz inequality to write
\begin{equation}\label{eq:diff.hor}
|u(x',\ell)-u(y',\ell)|^2\le \left(\int_0^\ell |\nabla_{x'} u(x'+s\xi,\ell)|\,ds\right)^2\le \ell \int_0^\ell |\nabla_{x'} u(x'+s\xi,\ell)|^2\,ds.
\end{equation}
Plugging this into the definition of $\mathcal S_2$ and applying the polar change of variable $y'=x'+\ell \xi\to (\ell,\xi)$ yields
\begin{align}
\label{eq:S2.polar.1}
\mathcal S_2={}&\int_0^\infty\int_{\mathbb S^{d-2}}\int_{\mathbb{R}^{d-1}}\frac{|u(x',\ell)-u(x'+\ell\xi,\ell)|^2}{\ell^2}\,dx'd\xi d\ell\\
\label{eq:S2.polar.2}
\le{}&\int_0^\infty\int_{\mathbb S^{d-2}}\int_{\mathbb{R}^{d-1}}\frac{1}{\ell}\int_0^\ell |\nabla_{x'} u(x'+s\xi,\ell)|^2\,dsdx'd\xi d\ell\\
\label{eq:S2.polar.3}
={}&\int_0^\infty\int_{\mathbb S^{d-2}}\int_{\mathbb{R}^{d-1}} |\nabla_{x'} u(x',\ell)|^2\,dx'd\xi d\ell,
\end{align}
where the conclusion follows applying a change of variable $x'+s\xi\to x'$ in the integral over $x'$, and noticing that the resulting integrand no longer depends on $s$. Since the integrand do not depend on $\xi$, we infer that $\mathcal S_2\le |\mathbb S^{d-2}|_{d-2}\norm{L^2(\Omega)}{\nabla_{x'} u}^2$, which is the expected bound.

\subsection{Notations}\label{sec:notations.trace}

In the rest of Section \ref{sec:proof.trace.cube}, we consider that $\Omega$ is a unit cube, and we only deal with traces of discrete functions on one of its sides. Consequently, $\partial\Omega$ only refers to this side, and $\Fhb$ (resp.~$\FFhb$) only to the mesh faces (resp.~pairs of distinct faces) on that side.
The proof of the discrete trace inequality relies on various partitions of the sets of faces and cells, which we introduce here together with some general notations.

For $x\in \dom$ we denote by $p(x)$ the orthogonal projection of $x$ on $\partial \dom$. If $Z$ is a face or cell of the mesh, we write $x_Z$ for the centroid of $Z$. We will usually denote a cell in $\Th$ with $t$, a boundary face in $\Fhb$ with $f$ and an interior face in $\Fhi$ with $\fint$. This distinct notation for interior of boundary faces is used to enhance the legibility of the proofs.

For $l\ge 1$, the set of pairs of boundary faces that are \emph{approximately at distance $lh$ of each other} is:
\begin{equation}\label{Wl}
  \begin{aligned}
    \W_l \doteq& \{(\fbou ,\fbou' ) \in \FFhb: x_{\fbou'} \in \An{r_1}{r_2}{x_{\fbou }} \; \mbox{with} \; r_1 = lh \; \mbox{and} \; r_2 = (l-1) h \}\\
    =& \{(\fbou ,\fbou' )\in\FFhb:(l-1)h\le \dffp <lh\},
  \end{aligned}
\end{equation}
where $\An{r_1}{r_2}{x}$ is the annulus on $\partial \dom$ centered at $x$ and with radii $r_1$ and $r_2$ (such that $r_1 > r_2$). 
For a fixed $\fbou \in\Fhb$, the slice of $\W_l$ at $\fbou $ (i.e., the faces $\fbou'$ that are \emph{approximately at distance $lh$ of $f$}) is
\begin{align}\label{Rfl}
  \W_{l\fbou} \doteq \{\fbou' \in \Fhb: (\fbou ,\fbou' )\in\W_l\}.
\end{align}
In the same way as we used polar changes of variables, in Section \ref{sec:trace.continuous}, to gather integrals over $y'$ according to the distance $\ell=|x'-y'|$ of $y'$ to a given $x'$, the set $\W_{l\fbou}$ will allow us to gather sums over $\fbou'$ according to their approximate distance $\simeq lh$ from a given face $\fbou$.

For each $\fbou \in\Fhb$, we denote the set of cells \emph{at the vertical of $\fbou$} by
\begin{equation}\label{def:set.above}
  \Above_{\fbou }=\{t\in\Th:\text{$x_f+\mathbb{R}^+\mathbf{n}_{\partial \dom}$ intersects $t$}\},
\end{equation}
where $\mathbf{n}_{\partial \dom}$ is the unit normal to $\partial\dom$ pointing inside $\dom$.

If $t\in\Th$, let $\delta_t=|x_t-p(x_t)|$ be the distance between $x_t$ and $\partial \dom$.
The set $\mathcal{T}_h$ is partitioned into layers according to this distance: for $m\ge 0$, the $m$-th layer is: 
\begin{equation}\label{set:Layers}
  \La_{m} = \{ t \in \Th: m h \leq \delta_t < (m+1) h\}.
\end{equation}
Summing over these layers will give us a discrete equivalent of integrating in the direction orthogonal to $\partial\dom$, as in \eqref{eq:diff.vert}.

Take $l\ge 1$ and let $(\fbou ,\fbou' )\in\W_l$. The set of cells \emph{between $f$ and $f'$ at height $\dffp$} is defined by (see also Figure \ref{fig:Slicing-6-7}-(A))
\begin{align}\label{set:Iffp}
  \Iffp \doteq \{ t \in \Th: t \mbox{ is intersected by the segment } [x_f + \dffp \textbf{n}_{\partial \Omega}, x_{f'} + \dffp \textbf{n}_{\partial \Omega}] \}.
\end{align}
If $t\in\Iffp$ then $|\dffp-\delta_t|\le h$. So, by \eqref{Wl}, $(l-2)h\le \delta_t\le (l+1)h$ and thus
\begin{align}\label{Height.Ca}
  \Iffp\subset \La_{l-2}\cup\La_{l-1}\cup\La_l.
\end{align}
For $s = 1, \dots, l+1$, the set 
\begin{align}\label{set:HorizontalLayers}
  \Ca = \left\{t \in \Iffp:  (s-1) h \leq |p(x_t)-x_f| < s h\right\},
\end{align}
collects the cells in $\Iffp$ according to their distance to $f$ parallel to $\partial \dom$ (see Figure \ref{fig:Slicing-6-7}-(B)). Summing over $s$ quantities of the form $\sum_{t\in\Ca}\bullet$ will give us a discrete equivalent of integrating in a parallel direction to $\partial\dom$, as in \eqref{eq:diff.hor}. 

We note that, if $t\in\Iffp$ then $x_t$ is within distance $h$ of $[x_f+\dffp\mathbf{n}_{\partial \dom},x_{f'}+\dffp\mathbf{n}_{\partial \dom}]$ and thus, after projection on $\partial \dom$, $p(x_t)$ is within distance $h$ of $[x_f,x_{f'}]$. As a consequence, $|p(x_t)-x_f|\le \dffp+h< (l+1)h$ since $(f,f')\in\W_l$. Hence, $(\Ca)_{s = 1, \dots, l+1}$ partitions $\Iffp$:
\begin{equation}\label{eq:Iffp.in.Ca}
  \Iffp = \bigsqcup_{s=1}^{l+1}\Ca.
\end{equation}

\begin{figure}[h]%
  \begin{tabular}{cc}
    {\includegraphics[width=.45\linewidth]{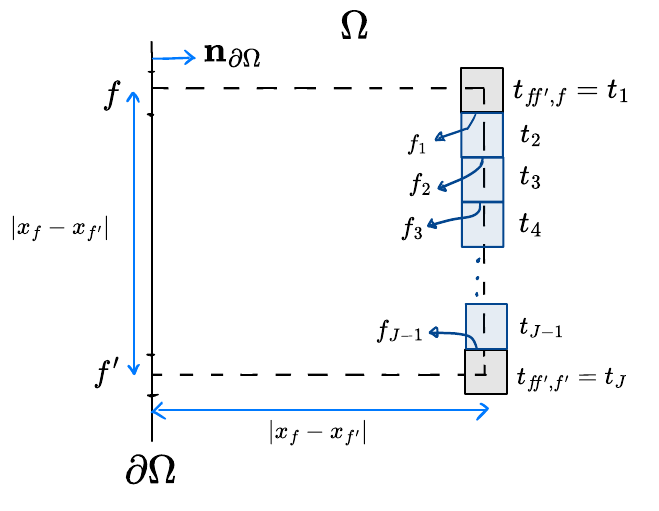} } &
    {\includegraphics[width=.45\linewidth]{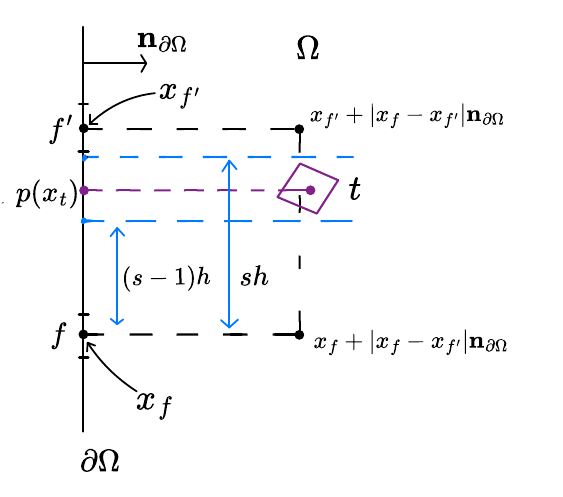} }\\
    (A) & (B)
  \end{tabular}
  \caption{(1) Cells between two boundary faces at a certain height ($t_1,t_2,\ldots,t_J$ are the cells in $\Iffp$). (B) Partitioning of these cells ($t$ is an example of a cell in $\Ca$).}
  \label{fig:Slicing-6-7}
\end{figure}

\subsection{Proof of the discrete trace inequality on the side of a cube}\label{sec:proof.trace}

To prove \eqref{eq:trace}, we need to handle the two terms appearing in the discrete $H^{1/2}$-seminorm \eqref{eq:def.disc.Hhalf}. Section \ref{sec:trace.local.contributions} deals with the first one, measuring the local variations of polynomials inside each face, while Section \ref{sec:trace.long.range} covers the other term, measuring the long-range interactions between averages. In several places, estimating these terms requires technical lemmas; these lemmas are stated when needed, and their proofs are postponed to Section \ref{sec:proof.technical.trace}
(in which the lemmas are also re-stated for ease of reference).

\subsubsection{Local contributions}\label{sec:trace.local.contributions}

In this section, we prove the following bound on the first term appearing in the $H^{1/2}$-seminorm of the trace of hybrid functions:
  \begin{equation}\label{eq:trace.local.contributions}
  \sum_{f \in \Fhb} h_f^{-1} \norm{L^2(f)}{v_f - \ol{v}_f}^2 \lesssim \seminorm{1,h}{\ul{v}_h}^2 \qquad \forall \ul{v}_h \in \Uh.
  \end{equation}

The local contribution $h_f^{-1} \norm{L^2(f)}{v_f - \ol{v}_f}^2$ to the discrete $H^{1/2}$-seminorm  measures the variation in $f$ of the hybrid function; it does not have any equivalent at the continuous level, since this scale is not directly apparent in the integrals. We bound this contribution using the local discrete $H^1$-norm $\seminorm{1,t}{\ul{v}_t}$, where $t$ is the cell neighbouring $f$, by considering two components: the internal (high-frequency at the scale of the cell) variations in $t$, and the difference between the lowest-order projections in $t$ and $f$ (low frequency at the scale of the cell). The following lemma gives a bound on the latter quantity.

\begin{restatable}{lemma}{lemlocalapprox}\label{lem:local.tf-approximation}
%\begin{lemma}\label{lem:local.tf-approximation}
  Let $t\in\Th$ and $\ul{v}_t =  (v_t,(v_f)_{f\in\mathcal{F}_t}) \in \mathbb{P}_k(t)\times \bigtimes_{f \in \mathcal F_t} \mathbb{P}_k(f)$. Then, for all $f\in\mathcal F_t$,
  \begin{align}
    \norm{L^2(f)}{\lproj{f}{0} v_f - \lproj{t}{0} v_t} &\lesssim h_t^{\frac{1}{2}} \seminorm{1,t}{\ul{v}_t}, \label{lem:local-tf-approximation.a} \\
    |\lproj{f}{0} v_f - \lproj{t}{0} v_t| &\lesssim h_t^{\frac{2-d}{2}} \seminorm{1,t}{\ul{v}_t}, \label{lem:local-tf-approximation.b}
  \end{align}
  where, for $X\in\{t,f\}$, $\lproj{X}{0}: L^2(X) \rightarrow \mathbb{P}_0(X)$ is the $L^2$-orthogonal projection onto constant functions in $X$.
%\end{lemma}
\end{restatable}

To prove \eqref{eq:trace.local.contributions}, take $f \in \Fhb$ and $t\in\Th$ such that $f \in\mathcal F_t$, and use triangle inequalities to write
\begin{align*}
  \norm{L^2(f)}{v_f - \ol{v}_f}^2 &\lesssim \norm{L^2(f)}{v_f - v_t}^2 + \norm{L^2(f)}{v_t - \lproj{t}{0} v_t}^2 + \norm{L^2(f)}{\lproj{t}{0} v_t - \lproj{f}{0} v_f}^2 \nonumber\\
  \overset{\eqref{lem:local-tf-approximation.a}}&\lesssim \norm{L^2(f)}{v_f - v_t}^2 + h_t \norm{L^2(t)}{\nabla v_t}^2 + h_t \seminorm{1,t}{\ul{v}_t}^2\\
  &\lesssim h_f \seminorm{1,t}{\ul{v}_t}^2.
\end{align*}
In the second inequality, we have additionally invoked the approximation property of $\lproj{t}{0}$ (see \cite[Theorem 1.45]{di-pietro.droniou:2020:hybrid}), while the conclusion comes from the definition \eqref{eq:def.disc.H1} of $\seminorm{1,t}{{\cdot}}$ together with the mesh regularity property that yields $h_t\lesssim h_f$. The bound \eqref{eq:trace.local.contributions} follows by dividing by $h_f$ and summing over $f$ (noticing that this implies, in the right-hand side, a sum over boundary cells $t$, which is smaller than a sum over all cells).

\subsubsection{Long-range term}\label{sec:trace.long.range}

In this section, we prove the bound on the second term in the discrete $H^{1/2}$-seminorm of the trace of hybrid functions:
\begin{equation}\label{eq:trace.long.range}
  \mathcal S\doteq\sum_{(f,f')\in\FFhb} |f|_{d-1} |f'|_{d-1} \frac{|\ol{v}_f - \ol{v}_{f'}|^2}{\dffp^d} \lesssim \seminorm{1,h}{\ul{v}_h}^2 \qquad \forall \ul{v}_h \in \Uh.
\end{equation}

For each pair $(f,f')\in\FFhb$, we select cells $t_{\ffp,f}\in\Above_f$ and $t_{\ffp,f'}\in\Above_{f'}$ that respectively contain $x_f+\dffp\mathbf{n}_{\partial \dom}$ and $x_{f'}+\dffp\mathbf{n}_{\partial \dom}$ (see Figure \ref{fig:Slicing-6-7}-(A)) and
we rewrite the difference $\ol{v}_f - \ol{v}_{f'}$ in a similar way as \eqref{continuous:decomp.u.minus.u} at the continuous level:
\begin{align*}
  \ol{v}_f - \ol{v}_{f'} = \ol{v}_f - \ol{v}_{t_{\ffp,f}} + \ol{v}_{t_{\ffp,f}} - \ol{v}_{t_{\ffp,f'}} + \ol{v}_{t_{\ffp,f'}} - \ol{v}_{f'},
\end{align*}
where $\ol{v}_t = \lproj{t}{0} v_t$.
Plugging this into $\mathcal S$ and using triangle inequalities leads to $\mathcal{S} \lesssim \mathcal{S}_1 + \mathcal{S}_2 + \mathcal{S}_3$, these quantities being the discrete versions of \eqref{continuous:S1S2S3}:
\begin{align*}
  \mathcal{S}_1 &\doteq \sum_{(f,f')\in\FFhb} |f|_{d-1} |f'|_{d-1} \frac{|\ol{v}_f - \ol{v}_{t_{\ffp,f}}|^2}{\dffp^d}, \\
  \mathcal{S}_2 &\doteq \sum_{(f,f')\in\FFhb} |f|_{d-1} |f'|_{d-1} \frac{|\ol{v}_{t_{\ffp,f}} - \ol{v}_{t_{\ffp,f'}}|^2}{\dffp^d}, \\
  \mathcal{S}_3 &\doteq \sum_{(f,f')\in\FFhb} |f|_{d-1} |f'|_{d-1} \frac{|\ol{v}_{f'} - \ol{v}_{t_{\ffp,f'}}|^2}{\dffp^d}=\mathcal S_1,
\end{align*}
the last equality following from the symmetric role that $f$ and $f'$ play. The bound \eqref{eq:trace.long.range} therefore follows if we can show that $\mathcal S_1\lesssim \seminorm{1,h}{\ul{v}_h}^2$ and $\mathcal S_2\lesssim \seminorm{1,h}{\ul{v}_h}^2$.

\medskip

\textit{\ul{Bound on $\mathcal{S}_1$}}: 
For $(f,f')\in\W_l$, considering the cells and faces intersected by the segment $[x_f, x_f+\dffp\mathbf{n}_{\partial \dom}]$ (see Figure \ref{fig:Slicing-4-5}), we rewrite the difference $\ol{v}_f - \ol{v}_{t_{\ffp,f}}$ as follows:
$$
  \ol{v}_f - \ol{v}_{t_{\ffp,f}} = \ol{v}_f - \ol{v}_{t_1} + \ol{v}_{t_1} - \ol{v}_{f_1} + \ol{v}_{f_1} - \ol{v}_{t_2} + \dots + \ol{v}_{f_{N-1}} - \ol{v}_{t_N}.
$$
\begin{figure}%
    \includegraphics[width=.7\linewidth]{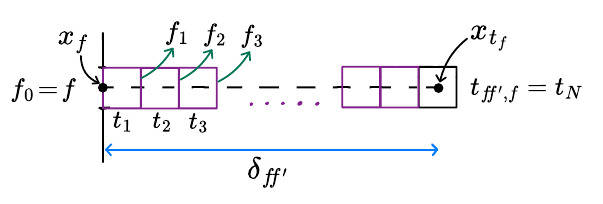} 
  \caption{Cells between a boundary face $f$ and the cell $t_{\ffp,f}$ at distance $\dffp$ above $f$.}
  \label{fig:Slicing-4-5}%
\end{figure}
The bound \eqref{lem:local-tf-approximation.b} then implies
\begin{equation*}%\label{key}
  |\ol{v}_f - \ol{v}_{t_{\ffp,f}}| \leq \sum_{i=0}^{N-1} |\ol{v}_{f_i} - \ol{v}_{t_{i+1}}| + \sum_{i=1}^{N-1} |\ol{v}_{t_{i}} - \ol{v}_{f_{i}}|
  \lesssim \sum_{i=1}^{N} h_{t_i}^{\frac{2-d}{2}} \seminorm{1,t_{i}}{\ul{v}_{t_{i}}}.
\end{equation*}
We now split this sum according to the distance of $t$ from $\partial\dom$, as represented by the layers $(\La_m)_m$.
If $t$ is intersected by $[x_f, x_f+\dffp\mathbf{n}_{\partial \dom}]$ then $t\in\Above_f$ and, since $\dffp<lh$, we have $\delta_t\le (l+1)h$ and thus $t\in\bigcup_{m = 0}^{l} \La_m$. We thus obtain the equivalent of \eqref{eq:diff.vert}, with $m$ playing the role of $s$:
\begin{equation}\label{eq:est.vf.vtf}
  |\ol{v}_f - \ol{v}_{t_{\ffp,f}}| \lesssim \sum_{m = 0}^{l} \underbrace{ \sum_{t\in\La_m\cap\Above_f} h_t^{\frac{2-d}{2}}\seminorm{1,t}{\ul{v}_{t}}}_{\doteq a^{(m)}_f}.	  
\end{equation}
The mapping 
$$
\chi_l: (s_f)_{f \in \Fhb}\in\mathbb R^{\Fhb} \longrightarrow \left(\sum_{(f,f') \in \W_{l}} \frac{|f|_{d-1} |f'|_{d-1}}{\dffp^d} |s_f|^2\right)^{1/2}
$$
is a norm and, by \eqref{eq:est.vf.vtf},
\begin{align}
  \zeta_l &\doteq \left(\sum_{(f,f') \in \W_{l}} |f|_{d-1} |f'|_{d-1} \frac{|\ol{v}_f - \ol{v}_{t_{\ffp,f}}|^2}{\dffp^d}\right)^{1/2} \label{eq:zetal}\\
  &\lesssim \left(\sum_{(f,f') \in \W_{l}} \frac{|f|_{d-1} |f'|_{d-1}}{\dffp^d} \left(\sum_{m=0}^{l} a_f^{(m)}\right)^2\right)^{1/2} = \chi_l \left(\sum_{m=0}^{l} a^{(m)}\right), \label{eq:zetal2}
\end{align}
with $a^{(m)} = (a_f^{(m)})_{f \in \Fhb}$. Note that $\sum_{(f,f') \in \W_{l}}\bullet=\sum_{f\in\Fhb}\sum_{f'\in\W_{lf}}\bullet$, so the sum in \eqref{eq:zetal} can be interpreted as $\int_{\mathbb S^{d-2}}\int_{\mathbb R^{d-1}}$ in \eqref{eq:S1.polar} (the integral over the directions on the sphere being less obvious than in the continuous case because of the discrete mesh structure); consequently, the term inside the bracket in \eqref{eq:zetal2} is the discrete equivalent of \eqref{eq:S1.polar.2}.
The triangle inequality on the norm $\chi_l$ is the equivalent of the continuous formula \eqref{eq:norm.integral}:
$$
  \chi_l \left(\sum_{m=0}^{l} a^{(m)}\right) \leq \sum_{m=0}^{l} \chi_l \left(a^{(m)}\right).
$$
Plugged into \eqref{eq:zetal2}, this gives
\begin{equation}\label{eq:bound.zetal}
  \zeta_l \lesssim \sum_{m=0}^{l} \chi_l \left(a^{(m)}\right) = \sum_{m=0}^{l} \left(\sum_{(f,f') \in \W_{l}} \frac{|f|_{d-1} |f'|_{d-1}}{\dffp^d} |a_f^{(m)}|^2\right)^{1/2}.
\end{equation}
The Cauchy-Schwarz inequality gives, for any family $(\alpha_i)_{i\in I}$ with $\#I\lesssim 1$,
\begin{equation}\label{CSinRn}
  \Big(\sum_{i \in I} \alpha_i\Big)^2 \leq \#I \sum_{i \in I} \alpha_i^2 \lesssim \sum_{i \in I} \alpha_i^2.
\end{equation}

We will then need some estimates, gathered in the following lemma, on the cardinality of the various sets of faces introduced so far.

\begin{restatable}[Estimates on the cardinality of sets of faces]{lemma}{lemcardinality}\label{lem:Cardinality}
%\begin{lemma}[Estimates on the cardinality of sets of faces]\label{lem:Cardinality} 
The following properties hold:
  \begin{enumerate}[label=(\roman*)]
    \item \label{lem:Lm.cardinality} For all $\fbou \in\Fhb$ and $m\ge 0$, $\#(\La_m\cap \Above_{\fbou })\lesssim 1$.
    \item \label{lem:Rfl.cardinality} For all $\fbou \in\Fhb$ and $l\ge 1$, $\# \W_{l \fbou } \lesssim l^{d-2}$.
    \item \label{lem:f.cardinality} For all $t\in\Th$, $\#\{\fbou \in\Fhb:t\in\Above_{\fbou }\}\lesssim 1$.
  \end{enumerate}
%\end{lemma}
\end{restatable}

Using Lemma \ref{lem:Cardinality}-\ref{lem:Lm.cardinality} to invoke \eqref{CSinRn} in the definition of $a_f^{(m)}$, we find
$$
|a_f^{(m)}|^2 \lesssim \sum_{t\in\La_m\cap\Above_f}h_t^{2-d} \seminorm{1,t}{\ul{v}_{t}}^2.
$$
Plugged into \eqref{eq:bound.zetal}, this yields
\begin{align}\label{eq:zetal.1}
  \zeta_l \lesssim{}& \sum_{m=0}^{l} \left(\sum_{(f,f') \in \W_{l}} \frac{|f|_{d-1} |f'|_{d-1}}{\dffp^d} 
  \sum_{t\in\La_m\cap\Above_f}
  h_t^{2-d}\seminorm{1,t}{\ul{v}_{t}}^2\right)^{1/2} \nonumber\\
  \lesssim{}& \sum_{m=0}^{l} \left(\sum_{(f,f') \in \W_{l}} \frac{1}{l^d} 
  \sum_{t\in\La_m\cap\Above_f}
  \seminorm{1,t}{\ul{v}_{t}}^2\right)^{1/2} \nonumber\\
  ={}& \sum_{m=0}^{l} \left(\sum_{f \in \Fhb} \sum_{f' \in \W_{lf}} \frac{1}{l^d}
  \sum_{t\in\La_m\cap\Above_f}
  \seminorm{1,t}{\ul{v}_{t}}^2\right)^{1/2},
\end{align}
where we have used, in the second inequality, Assumption \ref{assum:reg.mesh} and $\dffp \simeq lh$ (by definition \eqref{Wl} of $\W_l$) to write
\begin{equation}\label{eq:ffp.dffp}
  \frac{|f|_{d-1} |f'|_{d-1}}{\dffp^d} \; h_t^{2-d}\lesssim \frac{h^{d-1}h^{d-1}}{(lh)^d}h^{2-d}=\frac{1}{l^d}.
\end{equation}
Now an application of Lemma \ref{lem:Cardinality}-\ref{lem:Rfl.cardinality} in \eqref{eq:zetal.1} yields
\begin{align*}
  \zeta_l \lesssim  \frac{1}{l}\sum_{m=0}^{l} \left(\sum_{f \in \Fhb} \sum_{t\in\La_m\cap\Above_f}
  \seminorm{1,t}{\ul{v}_{t}}^2\right)^{1/2}.
\end{align*}
Rearranging the sums and using Lemma \ref{lem:Cardinality}-\ref{lem:f.cardinality} leads to
\begin{equation}\label{eq:zeta.pre.hardy}
\zeta_l \lesssim \frac{1}{l} \sum_{m=0}^{l} \left(\sum_{t \in \La_m} \seminorm{1,t}{\ul{v}_{t}}^2  
\sum_{f\in\Fhb \text{ s.t.~}t\in\Above_f} 1 \right)^{1/2} \lesssim \frac{1}{l} \sum_{m=0}^{l} \left(\sum_{t \in \La_m} \seminorm{1,t}{\ul{v}_{t}}^2 \right)^{1/2}.
\end{equation}
The right-hand side above is the discrete equivalent of the term in the bracket on the right of \eqref{eq:S1.polar.3}.
As in the continuous case, we need a (discrete) Hardy inequality to continue.

\begin{restatable}[Discrete Hardy inequality]{lemma}{lemhardy}\label{lem:Discrete-Hardy}
%\begin{lemma}[Discrete Hardy inequality]\label{lem:Discrete-Hardy}
  If $(r_{m})_{m = 0, \dots, L}$ are positive numbers and, for $l = 1, \dots, L$, $R_{l} \doteq \frac{1}{l} \sum_{m = 0}^{l} r_m$, then
  $$\sum_{l=1}^{L} R_{l}^2 \leq 18 \sum_{l=0}^{L} r_{l}^2.$$
%\end{lemma}
\end{restatable}

Setting $r_m = \left(\sum_{t \in \La_m} \seminorm{1,t}{\ul{v}_t}^2\right)^{1/2}$, the bound \eqref{eq:zeta.pre.hardy} gives $\zeta_l\le R_l$, and we infer from Lemma \ref{lem:Discrete-Hardy} that
\begin{align}\label{eq:Hardy}
  \sum_{l=1}^{L} \zeta_l^2 \lesssim \sum_{l=0}^{L} r_{l}^2,
\end{align}
where we can choose $L$ to ensure that the layers $(\La_m)_{m=0,\ldots,L}$ cover $\dom$. Recalling the definition \eqref{eq:zetal} of $\zeta_l$, the left-hand side of \eqref{eq:Hardy} is
$$
\sum_{l=1}^{L} \zeta_l^2 = \sum_{l=1}^{L} \sum_{(f,f') \in \W_{l}} |f|_{d-1} |f'|_{d-1} \frac{|\ol{v}_f - \ol{v}_{t_{\ffp,f}}|^2}{\dffp^d} 
= \sum_{f \in \Fhb} \sum_{l=1}^{L}  \sum_{f' \in \W_{lf}} |f|_{d-1} |f'|_{d-1} \frac{|\ol{v}_f - \ol{v}_{t_{\ffp,f}}|^2}{\dffp^d}.
$$
From the definition of $\W_{lf}$ in \eqref{Rfl} and our choice of $L$ we see that, for each $f\in\Fhb$, the family $(\W_{lf})_{l = 1, \dots, L}$ partitions $\Fhb\backslash\{f\}$. 
Hence, $ \sum_{l=1}^{L}  \sum_{f' \in \W_{lf}} \bullet =\sum_{f'\in\Fhb\backslash\{f\}} \bullet $ and
\begin{equation}\label{eq:Hardy.lhs}
  \sum_{l=1}^{L} \zeta_l^2 = \mathcal{S}_1.
\end{equation} 
In the right-hand side of \eqref{eq:Hardy}, we have
$$
\sum_{l=0}^{L} r_{l}^2 = \sum_{l=0}^{L} \left(\sum_{t \in \La_{l}} \seminorm{1,t}{\ul{v}_t}^2\right). 
$$
Recalling that the layers $(\La_{l})_{l = 0, \dots, L}$ partition $\Th$ by choice of $L$, we infer
\begin{equation}\label{eq:Hardy.rhs}
  \sum_{l=0}^{L} r^2_{l} = \sum_{t \in \Th} \seminorm{1,t}{\ul{v}_t}^2 = \seminorm{1,h}{\ul{v}_h}^2.
\end{equation}
Combining \eqref{eq:Hardy}, \eqref{eq:Hardy.lhs} and \eqref{eq:Hardy.rhs} gives the desired bound $\mathcal{S}_1\lesssim
\seminorm{1,h}{\ul{v}_h}^2$.

\medskip

\textit{\ul{Bound on $\mathcal{S}_2$}}: Let $(f,f')\in\W_l$. Analogous to $\mathcal{S}_1$, we have
\begin{equation}\label{eq:bound.S2.eq1}
  |\ol{v}_{t_{\ffp,f}} - \ol{v}_{t_{\ffp,f'}}| \leq \sum_{j=1}^{J-1} |\ol{v}_{t_{j+1}} - \ol{v}_{f_j}| + \sum_{j=1}^{J-1} |\ol{v}_{t_j} - \ol{v}_{f_j}|,
\end{equation}
where $(t_j)_{j=1,\ldots,N} \in \Iffp$ are the cells intersected by the segment $[x_f + \dffp  \mathbf{n}_{\partial \dom}, x_{f'} + \dffp  \mathbf{n}_{\partial \dom}]$ (see \eqref{set:Iffp} and Figure \ref{fig:Slicing-6-7}-(A)). Then, \eqref{lem:local-tf-approximation.b} yields
\begin{align}\label{key:S2}
  |\ol{v}_{t_{\ffp,f}} - \ol{v}_{t_{\ffp,f'}}| \lesssim  \sum_{t \in \Iffp}h_t^{\frac{2-d}{2}} \seminorm{1,t}{\ul{v}_t}.
\end{align}
The partition \eqref{eq:Iffp.in.Ca} allows us to gather the cells in the sum \eqref{key:S2} according to the distance of their projection on $\partial \dom$ to $f$ (as defined by $\Ca$) and to write
\begin{equation}\label{eq:olv.before.Ca}
|\ol{v}_{t_{\ffp,f}} - \ol{v}_{t_{\ffp,f'}}| \lesssim \sum_{s=1}^{l+1} \sum_{t \in \Ca} h_t^{\frac{2-d}{2}} \seminorm{1,t}{\ul{v}_t}
\le  (l+1)^{1/2}\left(\sum_{s=1}^{l+1} \left(\sum_{t \in \Ca} h_t^{\frac{2-d}{2}} \seminorm{1,t}{\ul{v}_t}\right)^2\right)^{1/2},
\end{equation}
where we have used the Cauchy--Schwarz inequality to conclude. We now need estimates on the cardinality of sets involving $\Ca$.

\begin{restatable}{lemma}{lemCa}\label{lem:Ca}
%\begin{lemma}\label{lem:Ca} 
It holds, for all $l\ge 1$ and $s=1,\ldots,l+1$,
  \begin{enumerate}[label=(\roman*)]
    \item \label{lem:Ca.1} For all $(f,f')\in\W_l$, $\# \Ca \lesssim 1$.
    \item \label{lem:Ca.2} For all $t\in\La_l$, $\#\{(f,f')\in\W_l:t\in\Ca\}\lesssim l^{d-2}$.
  \end{enumerate}
%\end{lemma}
\end{restatable}

Using Lemma \ref{lem:Ca}-\ref{lem:Ca.1} we can apply \eqref{CSinRn} to the innermost sum in \eqref{eq:olv.before.Ca} and get the discrete version of \eqref{eq:diff.hor}:
\begin{equation*}%\label{eq:S2.eq1}
  |\ol{v}_{t_{\ffp,f}} - \ol{v}_{t_{\ffp,f'}}|^2 \lesssim   l \sum_{s=1}^{l+1} \sum_{t \in \Ca} 	h_t^{2-d}\seminorm{1,t}{\ul{v}_t}^2.			
\end{equation*}
Multiplying by $ \frac{|f|_{d-1} |f'|_{d-1}}{\dffp^d}$, summing over $(f,f') \in \W_{l}$ and invoking \eqref{eq:ffp.dffp} we infer
\begin{equation}\label{eq:S2.eq2}
  \beta_{l} \doteq \sum_{(f,f') \in \W_{l}} |f|_{d-1} |f'|_{d-1} \frac{|\ol{v}_{t_{\ffp,f}} - \ol{v}_{t_{\ffp,f'}}|^2}{\dffp^d} \lesssim \frac{1}{l^{d-1}} \sum_{s=1}^{l+1} \sum_{(f,f') \in \W_{l}} \sum_{t \in \Ca} \seminorm{1,t}{\ul{v}_t}^2. 
\end{equation}
Note that $\beta_l$ is the discrete equivalent of the integrand under $\int_0^\infty\bullet d\ell$ in \eqref{eq:S2.polar.1}; indeed, we have
\begin{align}
  \sum_{l=1}^{L} \beta_l &= \sum_{l=1}^{L} \sum_{(f,f') \in \W_{l}} |f|_{d-1} |f'|_{d-1} \frac{|\ol{v}_{t_{\ffp,f}} - \ol{v}_{t_{\ffp,f'}}|^2}{\dffp^d} \nonumber\\
  &= \sum_{(f,f') \in \FFhb}  |f|_{d-1} |f'|_{d-1}\frac{|\ol{v}_{t_{\ffp,f}} - \ol{v}_{t_{\ffp,f'}}|^2}{\dffp^d} = \mathcal{S}_2,
\label{eq:S2.eq4}
\end{align}
where the second equality holds since $(\W_{l})_{l = 1, \dots, L}$ partitions $\FFhb$.

Using \eqref{Height.Ca} and swapping the sums over $t$ and over $(f,f')$ in \eqref{eq:S2.eq2}, we infer the following upper bound on $\beta_l$, corresponding to the integrand over $\ell$ in \eqref{eq:S2.polar.2}:
\begin{align*}
  \beta_{l} \lesssim \frac{1}{l^{d-1}} \sum_{s=1}^{l+1} \sum_{t \in \La_{l-2}\cup \La_{l-1}\cup\La_l} \seminorm{1,t}{\ul{v}_t}^2 \sum_{\substack{(f,f') \in \W_{l} \; s.t. \\ t \in \Ca}} 1\lesssim  \frac{1}{l}  \sum_{s=1}^{l+1} \sum_{t \in \La_{l-2}\cup \La_{l-1}\cup\La_l} \seminorm{1,t}{\ul{v}_t}^2,
\end{align*} 
where the conclusion follows from Lemma \ref{lem:Ca}-\ref{lem:Ca.2}. The term inside the sum over $s$ in the right-hand side does not depend on $s$ and, since $\frac{1}{l}  \sum_{s=1}^{l+1}1=\frac{l+1}{l}\lesssim 1$, this leads to the discrete version of the bound \eqref{eq:S2.polar.3}
$$
\beta_l \lesssim \sum_{t \in\La_{l-2}\cup \La_{l-1}\cup\La_l} \seminorm{1,t}{\ul{v}_t}^2.
$$
Since the layers partition $\Th$, we deduce that
\begin{align}\label{eq:S2.eq3}
  \sum_{l=1}^{L} \beta_l \lesssim \sum_{l=1}^{L} \sum_{t \in \La_{l-2}\cup \La_{l-1}\cup\La_l} \seminorm{1,t}{\ul{v}_t}^2 \le 3 \sum_{t \in \Th} \seminorm{1,t}{\ul{v}_t}^2 = 3\seminorm{1,h}{\ul{v}_h}^2.
\end{align}
The relations \eqref{eq:S2.eq4} and \eqref{eq:S2.eq3} lead to $\mathcal{S}_2 \lesssim \seminorm{1,h}{\ul{v}_h}^2$, as required. 

\subsection{Proof of the technical lemmas}\label{sec:proof.technical.trace}

We prove here the technical lemmas used in Section \ref{sec:proof.trace}. For legibility, these lemmas are re-stated before their proof.

\lemlocalapprox*

% \begin{proof}[Proof of Lemma \ref{lem:local.tf-approximation}]
\begin{proof}
Using the linearity, idempotency and $L^2(f)$-boundedness of $\lproj{f}{0}$ we get
\begin{align*}
  \norm{L^2(f)}{\lproj{f}{0} v_f - \lproj{t}{0} v_t} &= \norm{L^2(f)}{\lproj{f}{0}(v_f - \lproj{t}{0} v_t)}\\
  &\leq \norm{L^2(f)}{v_f - \lproj{t}{0} v_t} \nonumber\\
  &\leq \norm{L^2(f)}{v_f - v_t} + \norm{L^2(f)}{v_t - \lproj{t}{0} v_t} \nonumber\\
  &\lesssim \norm{L^2(f)}{v_f - v_t} + h_t^{\frac{1}{2}} \norm{L^2(t)}{\nabla v_t} \leq h_t^{\frac{1}{2}} \seminorm{1,t}{\ul{v}_t},
\end{align*}
where the third inequality follows from the approximation property of $\lproj{t}{0}$ (see \cite[Theorem 1.45]{di-pietro.droniou:2020:hybrid}), and the last inequality from the definition \eqref{eq:def.disc.H1} of $\seminorm{1,t}{\ul{v}_t}$. This concludes the proof of \eqref{lem:local-tf-approximation.a}. To deduce \eqref{lem:local-tf-approximation.b} from \eqref{lem:local-tf-approximation.a} we simply notice that $\lproj{f}{0} v_f - \lproj{t}{0} v_t$ is constant and that $|f|_{d-1}\simeq h_t^{d-1}$.
\end{proof}

\lemcardinality*

% \begin{proof}[Proof of Lemma \ref{lem:Cardinality}]
\begin{proof}
  \textit{Proof of \ref{lem:Lm.cardinality}}: The first property can readily be proved using the fact that $\mathcal{L}^m \cap \mathcal{V}_{\fbou }$ is contained in a cylinder with base an $h$-enlargement of $f$ and height $\simeq h$, which implies that $|\mathcal{L}^m \cap \mathcal{V}_{\fbou}|_d \lesssim h^d$. Since each $t\in \mathcal{L}^m \cap \mathcal{V}_{\fbou }$ has measure $\simeq h^d$ by quasi-uniformity of the mesh, we infer that $\#(\La_m\cap \Above_{\fbou })\lesssim \frac{|\mathcal{L}^m \cap \mathcal{V}_{\fbou}|_d}{h^d}\lesssim 1$.

  \medskip 
  \textit{Proof of \ref{lem:Rfl.cardinality}}: 
  We first note that the measure of an annulus on $\partial \dom$ with radii $r_1>r_2$ satisfies
  \begin{align}\label{MeasureOfAnnulus}
   |\An{r_1}{r_2}{z}|_{d-1}= \omega_{d-1} (r_1^{d-1} - r_2^{d-1}) \leq (d-1)\omega_{d-1} r_1^{d-2} (r_1 - r_2),
  \end{align}
  where $\omega_{d-1}$ is the measure of unit sphere in $d-1$ dimensions.   
  All faces in $\W_{l \fbou}$ have their centroid in $\An{lh}{(l-1)h}{x_f}$, and are thus contained in the annulus with same center
  and radii $r_1=(l+1)h$ and $r_2=(l-2)h$. By \eqref{MeasureOfAnnulus}, this annulus has measure $\lesssim ((l+1)h)^{d-2}\times 3h\lesssim l^{d-2}h^{d-1}$.
  Invoking then the quasi-uniformity of the mesh (which yields $|f'|_{d-1}\simeq h^{d-1}$ for all $f'\in\W_{l \fbou }$), we readily infer that  
  $\#\W_{l \fbou } \lesssim \frac{|\W_{l\fbou}|_{d-1}}{h^{d-1}}\lesssim \frac{l^{d-2} h^{d-1}}{h^{d-1}} \lesssim l^{d-2}$.  

  \medskip
  \textit{Proof of \ref{lem:f.cardinality}}: If $\fbou  \in \Fhb$ is such that $t\in\Above_{\fbou }$, then $x_f$ is contained in the projection on $\partial \dom$ of $t$, and thus $f$ is fully contained in the disc centered at $p(x_t)$ and of radius $2h$. Since this disc has measure $\lesssim h^{d-1}$, the same arguments as above give
  $
  \# \{f\in\Fhb:t\in\Above_f\} \lesssim\frac{h^{d-1}}{h^{d-1}}=1.
  $
\end{proof}

\lemhardy*

% \begin{proof}[Proof of Lemma \ref{lem:Discrete-Hardy}]
\begin{proof}
For $l \geq 1$ we have $\frac{1}{l^2} \leq 2 (\frac{1}{l}- \frac{1}{l+1})$.
So
\begin{align}\label{eq:Hardy.1}
  \sum_{l=1}^{L} R_{l}^2 &= \sum_{l=1}^{L} \frac{1}{l^2} \left(\sum_{m=0}^{l} r_m\right)^2 \nonumber \\
  &\leq 2 \sum_{l=1}^{L} \left(\frac{1}{l} - \frac{1}{l+1}\right) \left(\sum_{m=0}^{l} r_m\right)^2 \nonumber \\
  &\leq 	2 \sum_{l=1}^{L} \frac{1}{l} \left(\sum_{m=0}^{l} r_m\right)^2 - 2 \sum_{l=2}^{L+1} \frac{1}{l} \left(\sum_{m=0}^{l-1} r_m\right)^2 \nonumber \\
  &= 2 \sum_{l=2}^{L} \frac{1}{l} \left[\left(\sum_{m=0}^{l} r_m\right)^2 - \left(\sum_{m=0}^{l-1} r_m\right)^2\right] + 2 (r_0 +  r_1)^2- \frac{2}{L+1} \left(\sum_{m=0}^{L} r_m\right)^2\nonumber\\
  &\le 2 \sum_{l=2}^{L} \frac{1}{l} \left[\left(\sum_{m=0}^{l} r_m\right)^2 - \left(\sum_{m=0}^{l-1} r_m\right)^2\right] + 2 (r_0 +  r_1)^2. 
\end{align}
The identity $a^2 - b^2 = (a+b) (a-b)$ yields, for all $l\ge 2$,
$$
\left(\sum_{m=0}^{l} r_m\right)^2 - \left(\sum_{m=0}^{l-1} r_m\right)^2 = \left(\sum_{m=0}^{l} r_m + \sum_{m=0}^{l-1} r_m\right) r_{l}
\le lR_lr_l+(l-1)R_{l-1}r_l\le l(R_l+R_{l-1})r_l.
$$
Substituting the above bound in \eqref{eq:Hardy.1}, writing $2 (r_0+r_1)^2=2(r_0+r_1)r_0+2R_1r_1\le 3r_0^2+r_1^2+2R_1r_1$
and further simplifying yields
\begin{align*}
  \sum_{l=1}^{L} R_{l}^2 \leq 2\sum_{l=2}^L (R_l+R_{l-1})r_l + 3r_0^2+r_1^2+2R_1r_1
   \leq{}& 2\sum_{l=1}^L (R_l+R_{l-1})r_l + 3r_0^2+r_1^2\\
   \leq{}& \frac{1}{\alpha}\sum_{l=1}^L (R_l+R_{l-1})^2+\alpha\sum_{l=1}^L r_l^2 + 3r_0^2+r_1^2,
\end{align*}
where, in the second inequality, we have absorbed $2R_1r_1$ into the first sum on the right (as the term corresponding to 
$l=1$, setting $R_0=0$) and, in the conclusion, we have used the Young inequality $2ab\le \frac{1}{\alpha}a^2+\alpha b^2$
with $\alpha>0$ arbitrary. Selecting $\alpha\ge 3$ we can absorb $3r_0^2$ as the term $l=0$ in the second sum in the
right-hand side and, further using $(R_l+R_{l-1})^2\le 2R_l^2+2R_{l-1}^2$ and $r_1^2\le \sum_{l=0}^L r_l^2$, we infer
\[
  \sum_{l=1}^{L} R_{l}^2 \leq \frac{4}{\alpha}\sum_{l=1}^L R_l^2+(\alpha+1)\sum_{l=0}^L r_l^2.
\]
Fixing $\alpha=8$ and simplifying concludes the proof.
\end{proof}

\lemCa*

% \begin{proof}[Proof of Lemma \ref{lem:Ca}]~\\
\begin{proof}
  \textit{Proof of \ref{lem:Ca.1}}: The construction of $\Iffp$ implies that any $t \in \Iffp$ lies in a cylinder, parallel to $\partial\Omega$, of axis the $h$-enlargement of the segment $[x_f + \dffp \; \mathbf{n}_{\partial \dom}, x_{f'} + \dffp \; \mathbf{n}_{\partial \dom}]$ and base a $(d-1)$-dimensional disc of radius $2 h$. If moreover $t\in\Ca$ then $t$ actually lies in the section of this cylinder between lengths $(s-2)h$ and $(s+1)h$ along its axis. Since this section of cylinder has measure $\lesssim h^d$, the same reasoning as in the proof of Lemma \ref{lem:Cardinality}-\ref{lem:Lm.cardinality} yields $\#\Ca\lesssim 1$.
  
  \medskip	
  \textit{Proof of \ref{lem:Ca.2}}: Fix $t \in \La_l$, set $C_{tl}=\{(f,f')\in \W_l:t\in\Ca\}$, and let $C^1_{tl}=\{f\in\Fhb:\exists f'\in\Fhb \mbox{ s.t. }(f,f')\in C_{tl}\}$ be the projection of $C_{tl}$ on the first component of $\Fhb \times \Fhb$. If $f\in C_{tl}^1$ then, by the definition \eqref{set:HorizontalLayers} of $\Ca$ and since $f$ has diameter $\le h$ we have $f \subset \An{r_1}{r_2}{p(x_t)}$ with $r_1=(s+1)h$ and $r_2=(s-2)h$. The same arguments as in the proof of Lemma \ref{lem:Cardinality}-\ref{lem:Rfl.cardinality} then yield 
  \begin{equation}\label{eq:est.Ctl1}
    \# C_{tl}^1  \lesssim \frac{((s+1)h)^{d-2} h}{h^{d-1}} \lesssim s^{d-2}. 
  \end{equation}
  If $(f,f')\in C_{tl}$ then $t\in\Ca\subset \Iffp$ so the projection of $t$ on $\partial \dom$ intersects the line segment $[x_f, x_{f'}]$, and $x_{f'}$ is therefore in the cone $\mathfrak C$ of apex $x_f$ and opening defined by the projection of $t$ on $\partial \dom$; moreover, by \eqref{Wl}, $x_{f'}$ lies in the section of $\mathfrak C$ that is between distances $(l-1)h$ and $lh$ of $x_f$ (see Figure \ref{fig:slicing-8}). 
  \begin{figure}
    \includegraphics[width=.35\linewidth]{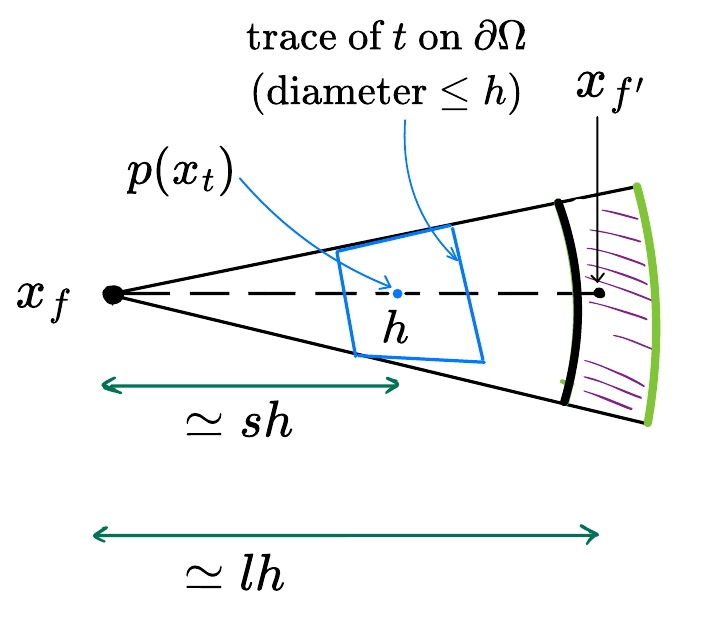}
    \caption{Cone $\mathfrak C$ of apex $x_f$ and opening defined by the trace of a cell $t$ on $\partial \Omega$.}
    \label{fig:slicing-8}
  \end{figure}
  Since $|p(x_t)-x_f|\simeq sh$ (as $t\in\Ca$) and the projection of $t$ on $\partial \dom$ is of diameter $\le h$, the disc centered at $x_f$ and of radius $sh$ intersects $\mathfrak C$ along a set of diameter $\simeq h$. By homothetie, the disc centered at $x_f$ and of radius $lh$ thus intersects $\mathfrak C$ along a set of diameter $\lesssim h\times \frac{lh}{sh}=h\frac{l}{s}$, see the green section in Figure \ref{fig:slicing-8}. The dashed region in this figure contains $x_f'$, so its $h$-enlargement -- both along the radius from $x_f$ and along the portions of spheres in black and green delimiting this region -- fully contains $f'$. The enlarged portions of spheres have diameter $\lesssim h\frac{l}{s}+h\lesssim h\frac{l}{s}$, since $l/s\ge l/(l+1)\ge 1/2$, so the enlargement of the dashed region has measure $\lesssim  \left(h \frac{l}{s}\right)^{d-2} h = \left(\frac{l}{s}\right)^{d-2} h^{d-1}$. Setting $C_{tlf}=\{f' \in \Fhb: (f,f')\in C_{tl}\}$ and using $|f'|_{d-1}\simeq h^{d-1}$ for all $f'\in\Fhb$, we infer that
  \begin{equation}\label{eq:est.Ctlf}
  \#C_{tlf} \lesssim \left(\frac{l}{s}\right)^{d-2}.
  \end{equation}
  In conclusion,
  $$
  \#C_{tl} = \sum_{\substack{(f,f') \in C_{tl}}} 1 = \sum_{f \in C_{tl}^1} \sum_{f' \in C_{tlf}} 1 \overset{\eqref{eq:est.Ctlf},\eqref{eq:est.Ctl1}}\lesssim s^{d-2} \times \left(\frac{l}{s}\right)^{d-2} = l^{d-2}. \qedhere
  $$
\end{proof}

\section{Discrete lifting property from the side of a cube}\label{sec:proof.lifting.cube}

The proof of the lifting property follows the same approach as the proof for the discrete trace inequality. We first consider the case of a cube and we lift a piecewise polynomial function defined on one of its sides. To do so, we follow the ideas of the continuous lifting done in \cite{Droniou_intsob:01}, which we briefly recall in Section \ref{sec:lifting.continuous}. Then, we introduce some notations for the discrete setting in Section \ref{sec:notations.lifting}, before giving the proof of the lifting from one side of $\dom$ into $\dom$ in Section \ref{sec:proof.lifting}. The case of a generic polytopal domain $\dom$ is treated in Section \ref{sec:trace.lift.polytopalcase}.

\subsection{Continuous lifting property}\label{sec:lifting.continuous}

We consider here that $\dom=\mathbb{R}^{d-1}\times (0,1)$ is a strip, we take $w\in H^{1/2}(\mathbb R^{d-1})$ and we construct $v\in H^1(\dom)$ whose trace on $\mathbb R^{d-1}\times \{0\}$ coincides with $w$ and such that $\seminorm{H^1(\dom)}{v}\lesssim \seminorm{H^{1/2}(\mathbb R^{d-1})}{w}$.

\medskip

To define the value of $v$ at $(x',x_d)\in\dom$, we consider an integral average of values of $w$ on a disc centered at $(x',0)$ and of radius $x_d$. This average is smoothed by a kernel to ensure regularity of $v$. We therefore start from a classical regularisation kernel: for $\epsilon>0$ and $z\in\mathbb R^{d-1}$ we set $\rho_\epsilon(z)=\epsilon^{-(d-1)}\rho(\epsilon^{-1}z)$ where $\rho\in C^1_c(\mathbb R^{d-1})$ is non-negative, is supported in the unit disc, and has integral equal to 1. We then set
\begin{equation}\label{eq:cont.lifting.v}
 v(x',x_d)=\int_{\mathbb R^{d-1}} w(y')\rho_{x_d}(x'-y')\,dy'.
\end{equation}
Classical properties of smoothing kernels ensure that $v(\cdot,x_d)\to w$ in $L^2(\mathbb{R}^{d-1})$ as $x_d\to 0$, which shows that the trace of $v$ on $\mathbb R^{d-1}\times \{0\}$ is $w$. We now have to estimate the $H^1(\dom)$-seminorm of $v$.

\medskip

We have
\begin{align}
\label{eq:lift.nabla.v}
\nabla v(x',x_d)={}&\int_{\mathbb R^{d-1}} w(y')\nabla_{(x',x_d)}(\rho_{x_d}(x'-y'))\,dy'\\
\label{eq:lift.nabla.v.2}
={}&\int_{\mathbb R^{d-1}} (w(y')-w(x'))\nabla_{(x',x_d)}(\rho_{x_d}(x'-y'))\,dy'
\end{align}
where, in the second line, we have used 
\[
\int_{\mathbb R^{d-1}}\nabla_{(x',x_d)}(\rho_{x_d}(x'-y'))\,dy'=\nabla_{(x',x_d)}\int_{\mathbb R^{d-1}}\rho_{x_d}(x'-y')\,dy'=\nabla_{(x',x_d)}1=0.
\]
It can easily be checked that
\[
|\nabla_{(x',x_d)}(\rho_{x_d}(x'-y'))|\lesssim \frac{1}{x_d^d}\left[\rho\left(\frac{x'-y'}{x_d}\right)+\left(1+\left|\frac{x'-y'}{x_d}\right|\right)|\nabla \rho|\left(\frac{x'-y'}{x_d}\right)\right]= \frac{1}{x_d^d}D\rho\left(\frac{x'-y'}{x_d}\right),
\]
where $z\in \mathbb R^{d-1}\to D\rho(z)\doteq\rho(z)+(1+|z|)|\nabla \rho|(z)$ is a compactly supported and bounded function.
Plugging this into \eqref{eq:lift.nabla.v.2}, squaring and using a Cauchy--Schwarz inequality yields
\begin{align}
|\nabla v(x',x_d)|^2\lesssim{}& \left(\int_{\mathbb R^{d-1}} |w(x')-w(y')| \frac{1}{x_d^{(d+1)/2}}D\rho\left(\frac{x'-y'}{x_d}\right)^{1/2} \frac{1}{x_d^{(d-1)/2}}D\rho\left(\frac{x'-y'}{x_d}\right)^{1/2}\,dy'\right)^2\nonumber\\
\lesssim{}& \int_{\mathbb R^{d-1}} |w(x')-w(y')|^2 \frac{1}{x_d^{d+1}}D\rho\left(\frac{x'-y'}{x_d}\right)\,dy'
 \times
 \int_{\mathbb R^{d-1}}  \frac{1}{x_d^{d-1}}D\rho\left(\frac{x'-y'}{x_d}\right)\,dy'.
\label{eq:lift.nabla.v.3}\end{align}
A change of variable $y'\to z=\frac{x'-y'}{x_d}$ gives
\begin{equation}\label{eq:g.change}
 \int_{\mathbb R^{d-1}}  \frac{1}{x_d^{d-1}}D\rho\left(\frac{x'-y'}{x_d}\right)\,dy'
=\int_{\mathbb R^{d-1}}D\rho(z)\,dz\lesssim 1.
\end{equation}
So, integrating \eqref{eq:lift.nabla.v.3} over $(x',x_d)\in \dom$, we obtain
\begin{equation}
\norm{L^2(\dom)}{\nabla v}^2\lesssim
 \int_{\mathbb R^{d-1}}  \int_{\mathbb R^{d-1}} |w(x')-w(y')|^2 \left(\int_0^1\frac{1}{x_d^{d+1}}D\rho\left(\frac{x'-y'}{x_d}\right)\,dx_d\right)\,dx'dy'.
\label{eq:lift.nabla.v.4}
\end{equation}
We then note that $D\rho(z)=0$ whenever $|z|\ge 1$, which implies
\begin{align}
\int_0^1\frac{1}{x_d^{d+1}}D\rho\left(\frac{x'-y'}{x_d}\right)\,dx_d&=
\int_{|x'-y'|}^1\frac{1}{x_d^{d+1}}D\rho\left(\frac{x'-y'}{x_d}\right)\,dx_d\nonumber\\
&\le \norm{L^\infty(\mathbb R^{d-1})}{D\rho}\int_{|x'-y'|}^1\frac{dx_d}{x_d^{d+1}}\lesssim \frac{1}{|x'-y'|^d}.
\label{eq:lift.estim.g}
\end{align}
Inserted into \eqref{eq:lift.nabla.v.4} this gives
\[
\norm{L^2(\dom)}{\nabla v}^2\lesssim \int_{\mathbb R^{d-1}}  \int_{\mathbb R^{d-1}} \frac{|w(x')-w(y')|^2 }{|x'-y'|^d}\,dx'dy'=\seminorm{H^{1/2}(\mathbb R^{d-1})}{w}^2,
\]
which concludes the proof of the continuous lifting property.

\subsection{Notations}\label{sec:notations.lifting}

In the rest of Section \ref{sec:proof.lifting.cube}, we come back to the same setting as in Section \ref{sec:notations.trace}: $\dom$ is a cube and $\partial\dom$ only refers to one of its sides (with $\Fhb$ and $\FFhb$ restricted to faces on this side).
We define $A_t$ as the region covered by the closures of the faces on $\partial\Omega$ that are within distance $\delta_t$ of $p(x_t)$ (see Figure \ref{fig:illustration-At}), that is:
\begin{equation}\label{def:At}
  A_t = \bigcup\{\mathrm{cl}(f): f \in \Fhb\text{ such that } \dist(p(x_t),f) \leq \delta_t \},
\end{equation}
where $\dist(z,f)$ is the distance between $z\in\partial\dom$ and $f\subset\partial\dom$, and $\mathrm{cl}(f)$ is the closure of $f$.

\begin{figure}[h]
  \centering
    \includegraphics[width=.4\linewidth]{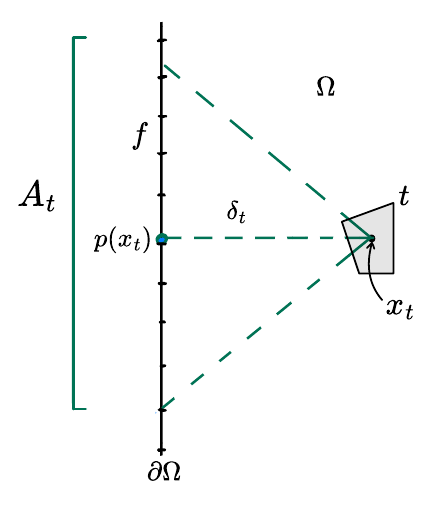}
  \caption{Illustration of $\A_t$.}
  \label{fig:illustration-At}
\end{figure}

The lifting of an element in $\Uhbb$ will be, as in the continuous case (see \eqref{eq:cont.lifting.v}), obtained by averaging face values over $A_t$. To this purpose we will use the following weights:
\begin{align}\label{eq:rho}
  \rho_t(f) \doteq \left\{
  \begin{aligned}
    &\frac{|f|_{d-1}}{|\A_t|_{d-1}} &\mbox{if $f \subset \A_t$},	\\
    &0 &\qquad \mbox{otherwise}.
  \end{aligned}
  \right.
\end{align}

For $\fint \in \Fhi$, similarly to $\delta_t$ we set $\delta_\fint=|x_\fint-p(x_\fint)|$.
The set of neighbouring cells in $\Th$ sharing a face $\fint$ is denoted by $\mathcal{T}_\fint$. By construction, there are only two cells in $\mathcal{T}_\fint$, which we label in the following as $t$ and $t'$.
 The absolute value of the difference of the associated weights, at a face $f \in \Fhb$, is
\begin{align}\label{eq:def.Deltatt}
  D_{g} \rho(f) \doteq |\rho_{t'}(f) - \rho_t(f)|\,,\quad\text{where $t,t'$ are the two cells in $\mathcal T_g$}. 
\end{align}
We will also need the union and symmetric difference of the sets $A_t$ on the sides of $g$:
$$
  \A_\fint \doteq \A_t \cup \A_{t'}\quad\text{ and }\quad \Delta_\fint \doteq A_t\Delta A_{t'}=(\A_t\cup \A_{t'})\backslash (\A_t\cap \A_{t'}).
$$

We fix a map $\projFace: \Fhi \rightarrow \Fhb$ that selects a boundary face at the vertical from $\partial\dom$ of an internal face (see Figure \ref{fig:proj-g}):
\begin{equation}\label{eq:def.projFace}
  \text{$\forall \fint\in\Fhi$, $\projFace(\fint)\in \Fhb$ is such that $p(x_\fint) \in \mathrm{cl}(\projFace(\fint))$}.
\end{equation}
Note that there may be several $f\in\Fhb$ such that $p(x_\fint) \in \mathrm{cl}(f)$ (if $p(x_\fint)$ belongs to the relative boundary of some faces), in which case $\projFace(\fint)\in \Fhb$ just corresponds to one of these $f$, selected arbitrarily.

\begin{figure}[htbp!]%
  \includegraphics[width=.5\linewidth]{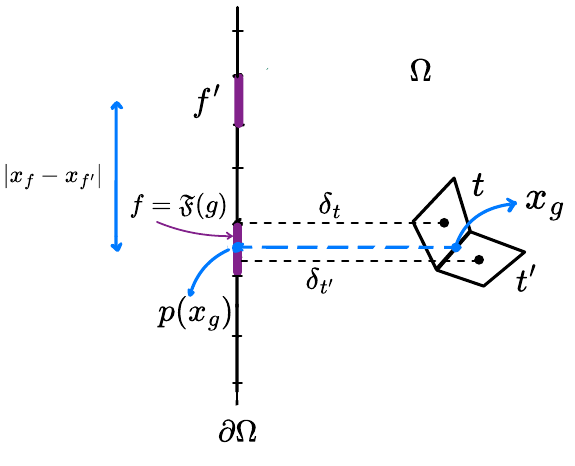}%
  \caption{Illustration of $\projFace$.}
  \label{fig:proj-g}
\end{figure}

Letting $\mathcal P(\Fhi)$ be the set of all subsets of $\Fhi$, we also define $\projFace^\dagger:\Fhb\to\mathcal P(\Fhi)$ that gathers all the internal faces that are \emph{approximately above} a selected boundary face: 
\begin{equation}\label{eq:def.projFacedagger}
  \text{$\forall\fbou\in\Fhb$, $\projFace^\dagger(\fbou)$ is the set of $\fint\in\Fhi$ such that $\projFace(\fint)=\fbou$}.
\end{equation}
We note that $\{ \projFace^\dagger(\fbou)  \}_{\fbou \in \Fhb}$ is a partition of $\Fhi$. 

\subsection{Proof of the discrete lifting property from the side of a cube}\label{sec:proof.lifting} 

The lifting $w_h\in \Uhbb\mapsto \mathcal L_h(w_h)=\ul{v}_h\in\Uh$ is defined the following way. Denoting as usual $w_h=(w_f)_{f \in \Fhb}$ and $\ul{v}_h=((v_t)_{t\in\Th},(v_f)_{f\in\Fh})$, we first define each cell value of $\ul{v}_h$ as a weighted average of boundary face values, similarly to \eqref{eq:cont.lifting.v}:
\begin{subequations}\label{eq:reconstructed}
  \begin{align}\label{eq:reconstructed.vt}
    v_t =  \sum_{f \in \Fhb} \ol{w}_f  \rho_t(f)\qquad\forall t\in\Th,
  \end{align}
  where we recall that $\ol{w}_f$ is the average of $w_f$ over $f$. Then, using these cell values, we define the face values of $\ul{v}_h$ by
  \begin{alignat}{2}
      \label{eq:reconstructed.vg}
      v_\fint &= \frac{v_{t} + v_{t'}}{2} &&\qquad \text{$\forall \fint \in \Fhi$, with $t,t'$ the cells in $\mathcal T_g$ (neighbouring $g$)}, \\
      \label{eq:reconstructed.vf}
      v_\fbou&= w_\fbou &&\qquad\text{$\forall f \in \Fhb$}.
  \end{alignat}
\end{subequations}
We obviously have $\tr(\ul{v}_h)=w_h$ so it remains to prove that the lifting estimate in \eqref{eq:lifting} holds.
\medskip

All components of $\ul{v}_h$ except boundary ones are constant, so for all $t\in\Th$ we have $\nabla v_t=0$ and $\norm{L^2(g)}{v_g-v_t}^2=|g|_{d-1}|v_g-v_t|^2$ for all $g\in\mathcal F_t\cap \Fhi$. Hence, 
\begin{align}
  \seminorm{1,h}{\ul{v}_h}^2 \overset{\eqref{eq:def.disc.H1}}& = \sum_{t \in \Th} \sum_{f \in \mathcal{F}_t} h_t^{-1} \norm{L^2(f)}{v_f-v_t}^2\nonumber\\
  &= \sum_{t \in \Th} \sum_{g \in \mathcal F_t\cap \Fhi} h_t^{-1} |g|_{d-1} |v_g - v_t|^2 + \sum_{t \in \Th} \sum_{f \in \mathcal F_t\cap\Fhb} h_t^{-1} \norm{L^2(f)}{v_f-v_t}^2\nonumber\\
  \overset{\text{Assum. \ref{assum:reg.mesh}}}&\simeq \sum_{g \in \Fhi} h^{d-2} (D_g v)^2 +  \sum_{f \in \Fhb} h^{-1}\norm{L^2(f)}{w_f-v_t}^2 = T_1 + T_2,
\label{eq:lifting.internal_boundary}
\end{align}
where, in the sum $T_1$, we have used $v_g - v_t \overset{\eqref{eq:reconstructed.vg}}= \frac{v_{t'} - v_{t}}{2}$ and set $D_g v \doteq  |v_{t'} - v_{t}|$ (with $t,t'$ the cells in $\mathcal T_g$) while, in the sum $T_2$, $t$ is the unique cell that contains $f \in \Fhb$. We now conclude by showing that $T_1\lesssim \tnorm{1/2,h}{w_{h}}^2$ and $T_2\lesssim \tnorm{1/2,h}{w_{h}}^2$.

\medskip

\textit{\ul{Bound on $T_1$}}: For $t,t'\in\Th$, using \eqref{eq:reconstructed.vt} we have (compare with \eqref{eq:lift.nabla.v} in the continuous setting)
\begin{equation}\label{eq:lift.T1.tt}
  v_{t'}-v_t = \sum_{f \in \Fhb} \ol{w}_f (\rho_{t'}(f)-\rho_t(f)).
\end{equation}
The faces are pairwise disjoint so $|A_t|_{d-1}=\sum_{f\subset A_t}|f|_{d-1}$ and
\begin{equation}\label{eq:sum.rho}
\sum_{f\in\Fhb}\rho_t(f)\overset{\eqref{eq:rho}}=\frac{1}{|A_t|_{d-1}}\sum_{f\subset A_t}|f|_{d-1}=1\qquad\forall t\in\Th.
\end{equation}
Hence,
$$
\sum_{f \in \Fhb} (\rho_{t'}(f)-\rho_t(f))=\sum_{f \in \Fhb} \rho_{t'}(f) - \sum_{f \in \Fhb} \rho_t(f) = 0.
$$
Taking an arbitrary $f'\in\Fhb$, multiplying the equation above by $\ol{w}_{f'}$ and subtracting it from \eqref{eq:lift.T1.tt} yields the discrete equivalent of \eqref{eq:lift.nabla.v.2}:
$$
v_{t'}-v_t= \sum_{f \in \Fhb} (\ol{w}_f - \ol{w}_{f'}) (\rho_{t'}(f)-\rho_t(f)).
$$ 
Applying this relation to $t,t'\in\mathcal T_\fint$ for some $\fint\in\Fhi$, taking $f'=\projFace(g)$ and recalling that $D_g\rho(f)=|\rho_{t'}(f)-\rho_t(f)|$, we infer
$$
  D_g v \leq  \sum_{f \in \Fhb} |\ol{w}_f - \ol{w}_{f'}|  D_g\rho(f).
$$
Analogously to \eqref{eq:lift.nabla.v.3}, the Cauchy-Schwarz inequality then yields
\begin{align}
  (D_g v)^2 \leq{}&  \left(\sum_{f \in \Fhb} |\ol{w}_f - \ol{w}_{f'}|^2 \; D_g\rho(f) \; \frac{h}{\delta_g}\right) \times \left(\sum_{f \in \Fhb} \frac{\delta_g}{h}D_g\rho(f)\right)\nonumber\\
  \lesssim{}&\sum_{f \in \Fhb} |\ol{w}_f - \ol{w}_{f'}|^2 \; D_g\rho(f) \; \frac{h}{\delta_g},
\label{eq:Dgv}
\end{align}
the second inequality following from Lemma \ref{lem:Dgrho} below, which establishes the discrete equivalent of \eqref{eq:g.change}.

\begin{restatable}[Estimate on $D_g\rho$]{lemma}{lemDgrho}\label{lem:Dgrho}
%\begin{lemma}[Estimate on $D_g\rho$]\label{lem:Dgrho}
For all $g\in\Fhi$, it holds
\begin{equation*}
  \sum_{f \in \Fhb} \frac{\delta_g}{h}D_g\rho(f) \lesssim 1.
\end{equation*}
%\end{lemma}
\end{restatable}

Plugging \eqref{eq:Dgv} into the definition in \eqref{eq:lifting.internal_boundary} of $T_1$, we obtain
\begin{equation}\label{eq:lift.bound.T1}
  T_1 \lesssim   \sum_{g \in \Fhi} \sum_{f \in \Fhb} |\ol{w}_f - \ol{w}_{f'}|^2 \; D_g\rho(f) \; \frac{h^{d-1}}{\delta_g}.
\end{equation}
Since $\{ \projFace^\dagger(\fbou')  \}_{\fbou' \in \Fhb}$ partitions $\Fhi$, we can write $\sum_{g \in \Fhi} \bullet  = \sum_{f' \in \Fhb} \sum_{g \in \projFace^{\dagger}(f')} \bullet $. Recalling that $f'=\projFace(g)$ in the sum in \eqref{eq:lift.bound.T1}, we obtain the equivalent of \eqref{eq:lift.nabla.v.4}:
\begin{equation*}%\label{eq:lift.bound.T1.2}
  T_1 \lesssim  \sum_{(f,f')\in\FFhb} \Big(|\ol{w}_f - \ol{w}_{f'}|^2  \sum_{g \in \projFace^{\dagger}(f')} \frac{h^{d-1}}{\delta_g} D_g\rho(f)\Big).
\end{equation*}
The desired estimate $T_1 \lesssim \tnorm{1/2,h}{w_{h}}^2$ then follows from Lemma \ref{lem:Qff} below (compare with \eqref{eq:lift.estim.g}).

\begin{restatable}{lemma}{lemQff}\label{lem:Qff}
%\begin{lemma} \label{lem:Qff} 
For all $(\fbou, \fbou') \in \FFhb$, the following holds: 
  $$
  \sum_{\fint \in \projFace^\dagger(f')} \frac{h^{d-1}}{\delta_\fint} D_\fint\rho(\fbou)  \lesssim  \frac{|\fbou |_{d-1} |\fbou' |_{d-1}}{\dffp^d}.
  $$
%\end{lemma} 
\end{restatable}
\medskip

\textit{\ul{Bound on $T_2$}}: In order to estimate $T_2,$ we introduce $\pm\ol{w}_f$ and use a triangle inequality together with $ \norm{L^2(f)}{\ol{w}_f-v_t}^2=|f|_{d-1}|\ol{w}_f-v_t|^2\lesssim h^{d-1}|\ol{w}_f-v_t|^2$ to write
\begin{align}
  T_2 &\lesssim \sum_{f \in \Fhb} h^{-1} \norm{L^2(f)}{w_f - \ol{w}_f}^2 + \sum_{f \in \Fhb} h^{d-2} |\ol{w}_f - v_t|^2 \nonumber\\
  \label{eq:list.est.T2}
  \overset{\eqref{eq:def.disc.Hhalf}}&\lesssim  \tnorm{1/2,h}{w_{h}}^2 + \sum_{f \in \Fhb} h^{d-2} |\ol{w}_f - v_t|^2.
\end{align}
Let $f\in\Fhb$ and $t$ be the cell that contains $f$ in its boundary (which is the cell appearing in the last sum above). By \eqref{eq:reconstructed.vt}, \eqref{eq:sum.rho} and the convexity of the square function,
$$
|\ol{w}_f - v_t|^2 =\left|\sum_{f'\in\Fhb}(\ol{w}_{f}-\ol{w}_{f'})\rho_t(f')\right|^2\leq \sum_{f' \in \Fhb} |\ol{w}_f - \ol{w}_{f'}|^2\rho_t(f')
\le \sum_{f' \subset \A_t\backslash f} |\ol{w}_f - \ol{w}_{f'}|^2,
$$
where, in the conclusion, we have used $\rho_t(f')=0$ if $f'\not\subset A_t$ and $\rho_t(f')\le 1$ for all $f'\in\Fhb$.
Since $t$ is a boundary cell, we have $\delta_t=|x_t-p(x_t)|\le h$. Recalling the definition \eqref{def:At} of $\A_t$, we infer that any $f'\subset A_t$ is within distance $\lesssim h$ of $x_t$, and thus within distance $\lesssim h$ of $f$. This gives $\dffp\lesssim h$ and thus, coming back to \eqref{eq:list.est.T2}, we obtain the desired estimate:
\begin{align*}
  T_2 &\lesssim \tnorm{1/2,h}{w_h}^2 + \sum_{f \in \Fhb} \sum_{f' \subset \A_t\backslash f} |f|_{d-1} |f'|_{d-1} \frac{|\ol{w}_f - \ol{w}_{f'}|^2}{\dffp^d} \times \frac{h^{d-2} \; \dffp^d}{|f|_{d-1} |f'|_{d-1}} \\
  \overset{\text{Assum. \ref{assum:reg.mesh}}}&\lesssim \tnorm{1/2,h}{w_{h}}^2 + \sum_{f \in \Fhb} \sum_{f' \subset A_t\backslash f} |f|_{d-1} |f'|_{d-1} \frac{|\ol{w}_f - \ol{w}_{f'}|^2}{\dffp^d} 
  \times \frac{h^{2d-2}}{h^{d-1} h^{d-1}} \\
  \overset{\eqref{eq:def.disc.Hhalf}}&\lesssim \tnorm{1/2, h}{w_h}^2.
\end{align*}

\subsection{Proof of the technical lemmas}\label{sec:proof.technical.lifting}

Throughout this section, $D(x, r)$ denotes the disc on $\partial\Omega$ centered at $x$ with radius $r$.
The next lemma establishes preparatory results to prove Lemma \ref{lem:Dgrho}.

\begin{lemma}[Estimates on $A_t$]\label{lem:At}
  For all $t \in \Th$ and $g\in\Fhi$, denoting by $t,t' \in \mathcal{T}_g$ the two cells on each side of $g$, the following properties hold:
  \begin{enumerate}[label=(\roman*)]
    \item \label{lem:At.0} $h \lesssim \delta_t$, $|\delta_t-\delta_{t'}|\le 4h$ and $\delta_t\simeq \delta_{t'}\simeq \delta_g$.
    \item \label{lem:At.1} $|\A_t|_{d-1} \simeq \delta_t^{d-1}$.
    \item \label{lem:At.2} $ |\Delta_g|_{d-1}\lesssim h\delta_g^{d-2}$.
    \item \label{lem:At.3} If $f \subset \A_t \cap \A_{t'}$, then $D_g\rho(f) \lesssim \frac{h^d}{\delta_g^d}$.
    \item \label{lem:At.4} For any $f\in\Fhb$ we have $D_g\rho(f) \lesssim \frac{h^{d-1}}{\delta_g^{d-1}}$.  
  \end{enumerate}
\end{lemma}

\begin{proof}[Proof of Lemma \ref{lem:At}]~\\
  \textit{Proof of \ref{lem:At.0}}: 
  By definition of the centroids, for any $t\in\Th$ there exists $r \gtrsim h$ such that the ball centered at $x_t$ with radius $r$ is contained in $t$. This ball does not intersect $\partial \dom$ and thus $r \le \delta_t$, from which we infer that	$h\lesssim \delta_t$. Since $t$ and $t'$ share the face $g$, we have $|x_t-x_{t'}|\le |x_t-x_g|+ |x_g-x_{t'}|\le 2h$.
  The orthogonal projection is 1-Lipschitz so 
  \begin{equation}\label{eq:proj.tt'}
    |p(x_t)-p(x_{t'})|\le 2h.
  \end{equation}
  Hence, 
  \[
  \delta_t=|x_t-p(x_t)|\le |x_t-x_{t'}|+|x_{t'}-p(x_{t'})|+|p(x_{t'})-p(x_t)|\le 2h+\delta_{t'}+2h.
  \]
  Swapping the roles of $t$ and $t'$, we conclude that $|\delta_t-\delta_{t'}|\le 4h$, and the relation $\delta_{t'}\simeq \delta_t$ is then obtained using the fact that $h\lesssim \delta_t$ and $h\lesssim \delta_{t'}$. To get $\delta_g\lesssim\delta_t$, we write
  \[
  \delta_g=|x_g-p(x_g)|\le |x_g-x_t|+|x_t-p(x_t)|+|p(x_t)-p(x_g)|\le 2h+\delta_t\lesssim \delta_t,
  \]
  where we have used the fact that $p$ is 1-Lipschitz in the second inequality and that $|x_g-x_t|\le h$. The reverse inequality is obtained by swapping the roles of $t$ and $g$, and by noticing that $h\lesssim \delta_g$, thanks to the mesh regularity assumption  \cite[Definition 1.9]{di-pietro.droniou:2020:hybrid} and the fact that $g$ contains a disc centered at $x_g$ and of radius $\gtrsim h$.
  
  \medskip
  \textit{Proof of \ref{lem:At.1}}: 
  For any $f\subset\A_t$, there is $x\in \mathrm{cl}{(f)}$ such that $|x-p(x_t)| \le \delta_t$. Since $\mathrm{diam}(f) \leq h$, $f$ is contained in the disc $D(p(x_t), \delta_t+h)$ on $\partial \dom$ (which has dimension $d-1$). Using the bound $h \lesssim \delta_t$ in \ref{lem:At.0}, $A_t$ is therefore contained in a disc centered at $p(x_t)$ and of radius $\lesssim\delta_t$, which implies
  \begin{equation}\label{eq:upper.bound.At}
    |A_t|_{d-1}\lesssim \delta_t^{d-1}.
  \end{equation}
  
  The disc $D(p(x_t), \delta_t)$ on $\partial \dom$ is covered by all the faces $f\in\Fhb$ that intersect it. By definition, these faces are included in $\A_t$, so this disc is contained in $A_t$ and $|A_t|_{d-1}\ge |D(p(x_t), \delta_t)|_{d-1}\gtrsim \delta_t^{d-1}$. Combined with \eqref{eq:upper.bound.At} this proves \ref{lem:At.1}.
  \medskip
  
  \textit{Proof of \ref{lem:At.2}}:   
  As shown above, $D(p(x_t), \delta_t) \subset A_t$. Let $x \in D(p(x_t), \delta_t -6h)$ and note that
  \begin{equation}\label{eq:x.in.Atp}
    |p(x_{t'})-x| \leq  |p(x_{t'}) - p(x_t)|+|p(x_t) - x| \le 2h + \delta_t - 6h 
    \leq \delta_t - 4h \leq \delta_{t'},
  \end{equation}
  where we have used \eqref{eq:proj.tt'} and the definition of the disc $D(p(x_t), \delta_t -6h)$ in the second inequality, and $|\delta_t - \delta_{t'}| \leq 4h$ from \ref{lem:At.0} in the last one. Letting $f\in\Fhb$ be a face that contains $x$ in its closure, \eqref{eq:x.in.Atp} shows that $f\subset\A_{t'}$. Hence, $D(p(x_t), \delta_t -6h) \subset A_{t'}$ and thus $D(p(x_t), \delta_t -6h) \subset A_t\cap A_{t'}$.
  
  We also saw that $A_t \subset D(p(x_t), \delta_t + h)$ in the proof of \eqref{eq:upper.bound.At}. Following the same arguments as above, we infer that $A_{t'} \subset D(p(x_t), \delta_t + 7h)$. Hence, $A_t\cup A_{t'} \subset D(p(x_t), \delta_t + 7h)$.
  
  These relations show that $\Delta_g=(A_t\cup A_{t'})\backslash (A_t\cap A_{t'})$ is contained in the annulus $\An{r_1}{r_2}{p(x_t)}$ with $r_1=\delta_t+7h$ and $r_2=\delta_t-6 h$. We have $r_1\lesssim \delta_g$ by \ref{lem:At.0} and $r_1-r_2=13h$, so the estimate \eqref{MeasureOfAnnulus} on the measure of $\An{r_1}{r_2}{p(x_t)}$ yields $|\Delta_g|_{d-1}\lesssim h\delta_g^{d-2}$.
  
  \medskip	
  \textit{Proof of \ref{lem:At.3}}: 
  For $f \subset \A_t \cap \A_{t'}$, the definition \eqref{eq:rho} of the weights yields
  \begin{equation}\label{eq:bound.Dg.init}
    D_g\rho(f) = \left|\frac{|f|_{d-1}}{|\A_t|_{d-1}} - \frac{|f|_{d-1}}{|\A_{t'}|_{d-1}}\right| = |f|_{d-1}\frac{\big| |\A_{t'}|_{d-1} - |\A_{t}|_{d-1}\big|}{|\A_{t'}|_{d-1}|\A_t|_{d-1}}.
  \end{equation}
  Since $\A_t\subset \Delta_g \cup \A_{t'}$, we have $|\A_t|_{d-1}\le |\Delta_g|_{d-1}+|\A_{t'}|_{d-1}$. Swapping the roles of $t$ and $t'$ and using \ref{lem:At.2} we infer $\big| |\A_t|_{d-1} - |\A_{t'}|_{d-1}\big| \leq |\Delta_g|_{d-1} \lesssim h\delta_g^{d-2}$.
  By \ref{lem:At.1} and \ref{lem:At.0}, we have $|\A_{t'}|_{d-1}\simeq |\A_{t}|_{d-1}\simeq \delta_g^{d-1}$. Plugging these estimates in \eqref{eq:bound.Dg.init} and using $|f|_{d-1}\lesssim h^{d-1}$ concludes the proof of \ref{lem:At.3}.
  
  \medskip
  \textit{Proof of \ref{lem:At.4}}: Direct consequence of the definitions \eqref{eq:def.Deltatt} and \eqref{eq:rho} of $D_g\rho(f)$ and $\rho_t(f)$, the triangle inequality, $|f|_{d-1}\lesssim h^{d-1}$ and \ref{lem:At.0}--\ref{lem:At.1}.
  
  \end{proof}
  
\lemDgrho*

\begin{proof}[Proof of Lemma \ref{lem:Dgrho}]  
  We write
  $$
    \sum_{f \in \Fhb} D_g\rho(f) = \sum_{f \subset \A_g \backslash \Delta_g} D_g\rho(f) + \sum_{f \subset \Delta_g} D_g\rho(f) + \sum_{f \in \Fhb,\;f\not\subset \A_g} D_g\rho(f).
  $$ 
  Recalling that $A_g=A_t\cup A_{t'}$ (with $t,t'$ the two cells on each side of $g$), the definitions \eqref{eq:def.Deltatt} of $D_g\rho(f)$ and \eqref{eq:rho} of $\rho_t(f)$ show that the last term in the above equality vanishes. Noting that $\A_g \backslash \Delta_g=\A_t\cap \A_{t'}$, an application of Lemma \ref{lem:At}-(\ref{lem:At.3},\ref{lem:At.4}) then yields
  \begin{align*}
    \sum_{f \in \Fhb} D_g\rho(f) &\lesssim  \sum_{f \subset \A_g \backslash \Delta_g} \frac{h^d}{\delta_g^d} + \sum_{f \subset \Delta_g} \frac{h^{d-1}}{\delta_g^{d-1}}\\
    \overset{\text{Assum. \ref{assum:reg.mesh}}}&\lesssim  \frac{h}{\delta_g^d}\sum_{f \subset \A_g \backslash \Delta_g} |f|_{d-1} + \frac{1}{\delta_g^{d-1}}\sum_{f \subset \Delta_g} |f|_{d-1}\\
    &\le  \frac{h}{\delta_g^d}|A_t|_{d-1}+\frac{1}{\delta_g^{d-1}}|\Delta_g|_{d-1}\lesssim \frac{h}{\delta_g},
  \end{align*}    
  where the conclusion follows from Lemma \ref{lem:At}-(\ref{lem:At.0},\ref{lem:At.1},\ref{lem:At.2}).
\end{proof}  
 
 The following lemma states intermediate results that will be useful to prove Lemma \ref{lem:Qff}.
 
 \begin{lemma}\label{lem:Cases}
  The following results hold for any $\fint\in\Fhi$ :
  \begin{enumerate}[label=(\roman*)] 
    \item \label{lem:Cases.1} If $\fbou  \subset A_\fint$, then $|x_{\fbou}-x_{\projFace(\fint)}|\lesssim \delta_{\fint}$.
    \item \label{lem:Cases.2-b} For all $\fbou, \fbou' \in \Fhb$, $\# \{ \fint \in \Fhi \, : \,  \fbou' = \projFace(\fint), \fbou \subset \Delta_\fint \}\lesssim 1$. 
  \end{enumerate}
\end{lemma}

\medskip

\begin{proof}[Proof of Lemma \ref{lem:Cases}]~\\
  \textit{Proof of \ref{lem:Cases.1}}: 
  For simplicity of notations, let us set $\fbou' \doteq \projFace(\fint)$. Given ${\fbou} \subset \A_\fint$ and $t \in \mathcal{T}_\fint$, we write
  \begin{align}\label{delff'}
    |x_\fbou - x_{\fbou'}| 
    \leq |x_\fbou - p(x_t)| + |p(x_t) -  x_{\fbou'}|. 	
  \end{align}
  The projection $p(x_{\fint})$ belongs to $\mathrm{cl}(\fbou')$ by definition of $\projFace$, so $|p(x_{\fint}) - x_{\fbou'}| \leq h$. We also have $|x_{t} - x_{\fint}| \leq h$, hence $|p(x_t) - p(x_{\fint})| \leq h$ since $p$ is 1-Lipschitz. These estimates show that
  \begin{equation}\label{eq:pxtxf'}
    |p(x_t) -  x_{\fbou'}|\le |p(x_t) - p(x_{\fint})| + |p(x_{\fint}) - x_{\fbou'}| \leq 2h.
  \end{equation}
  The condition $\fbou \subset A_\fint \subset A_t$, implies $\dist(p(x_t),\fbou) \leq \delta_t$. Moreover, $\mathrm{diam}(\fbou) \leq h$ and $x_{\fbou} \in \fbou$, so
  \begin{align}\label{eq:u.dist(pxt,xf)}
    |p(x_t) - x_{\fbou}| \leq \delta_t + h.
  \end{align}
  Plugging \eqref{eq:pxtxf'} and \eqref{eq:u.dist(pxt,xf)} into \eqref{delff'} and using Lemma \ref{lem:At}-\ref{lem:At.0} we conclude 
  \begin{align}\label{eq:u.dff'-a}
    \dffp \leq \delta_t + 3h \lesssim \delta_t\lesssim \delta_g.
  \end{align}
  
  \medskip
  \textit{Proof of \ref{lem:Cases.2-b}}. Let $\fbou, \fbou' \in \Fhb$ and $g \in \Fhi$ such that $\projFace(\fint) = f'$ and $f \subset \Delta_g$. Let $t, t'$ be the two cells in $\mathcal{T}_\fint$. By definition of the symmetric difference,
  $$
    \fbou \subset \Delta_g=(A_{t} \backslash A_{t'}) \cup (A_{t'} \backslash A_{t}).
  $$  
  We first consider the case $\fbou \subset A_t \backslash A_{t'}$. The condition $\fbou \not\subset \A_{t'}$ implies that $|p(x_{t'}) - x_{\fbou}| > \delta_{t'}$, and thus
  \begin{align*}
    \delta_{t'} < |p(x_{t'}) - x_{\fbou}| \leq |p(x_{t'}) - p(x_{t})| + |p(x_{t}) - x_{\fbou}| \overset{\eqref{eq:proj.tt'}}\leq 2h + |p(x_{t}) - x_{\fbou}|.
  \end{align*}
  Lemma \ref{lem:At}-\ref{lem:At.0} then gives
  \begin{align*}
    |p(x_{t}) - x_{\fbou}| \geq \delta_{t'} - 2h \geq \delta_{t} - 6h.
  \end{align*}
Recalling \eqref{eq:pxtxf'}, we are able to derive a lower bound for $\dffp$ as follows:
\begin{align*}
  \dffp \geq |x_f - p(x_t)| - |p(x_t) - x_{f'}| \geq \delta_t - 6h - 2h = \delta_t - 8h.
\end{align*}
Together with \eqref{eq:u.dff'-a}, this implies
\begin{align*}
  \delta_t - 8h \leq \dffp \leq \delta_t + 3h
\end{align*}
and thus, recalling the definition of $\delta_t$ and rearranging terms,
\begin{align*}
  \dffp - 3h \leq |x_{t} - p(x_t)| \leq \dffp + 8h.
\end{align*}
Moreover, by \eqref{eq:pxtxf'}, we have  $p(x_t) \in D(x_{f'}, 2h)$. So $x_t$ is in the cylinder vertical to $\partial\dom$, with base $D(x_{f'}, 2h)$ and between heights $\dffp - 3h$ and $\dffp + 8h$ from $\partial \dom$.
We can proceed analogously for $\fbou \subset  \A_{t'} \backslash \A_t$. As a consequence, any $g \in \Fhi$ such that $\projFace(\fint) = f'$ and $f \subset \Delta_g$ has a neighbouring cell $t$ such that $x_t$ lies in a cylinder of base a disc of radius $\simeq h$ and with height $\simeq h$. Such cells $t$ are entirely contained in the cylinder obtained by enlarging the base and height by $h$ and, reasoning as in the proof of Lemma \ref{lem:Cardinality}-\ref{lem:Lm.cardinality}, we obtain a bound $\lesssim 1$ on the number of these cells $t$. As $\#\mathcal F_t\lesssim 1$ by mesh regularity assumption, this leads to the required bound on the number of considered $\fint$.
\end{proof}

\lemQff* 

\begin{proof}[Proof of Lemma \ref{lem:Qff}]
  Let $Q_{\ffp}$ be the left-hand side of the inequality in the lemma. We first split this sum according to which set $\fbou $ belongs to: 	\begin{align}\label{eq:Qff}
    Q_\ffp &= \sum_{\substack{ \fint \in \projFace^\dagger (\fbou') \\ \text{s.t.} \; \fbou  \subset \A_\fint \setminus \Delta_\fint}} \frac{h^{d-1}}{\delta_\fint} D_\fint\rho(\fbou ) + \sum_{\substack{\fint \in \projFace^\dagger (\fbou' ) \\ \text{s.t.} \; \fbou  \subset \Delta_\fint }} \frac{h^{d-1}}{\delta_\fint} D_\fint\rho(\fbou ) + \sum_{\substack{\fint \in \projFace^\dagger (\fbou' ) \\ \text{s.t.} \; {\fbou  \in \Fhb,\,\fbou\not\subset A_\fint }}} \frac{h^{d-1}}{\delta_\fint} D_\fint\rho(\fbou )  \nonumber\\
    &\doteq Q_{1} + Q_{2} + Q_3.
  \end{align}  
  If $\fbou \in \Fhb$ and $\fbou\not\subset A_{\fint}$ the definitions \eqref{eq:def.Deltatt} and \eqref{eq:rho} imply that $D_\fint\rho(\fbou)$ vanishes, thus $Q_3=0$. For the sum in $Q_1,$ recalling that $\A_g\backslash\Delta_g=\A_t\cap \A_{t'}$ we apply Lemma \ref{lem:At}-(\ref{lem:At.0},\ref{lem:At.3}) then exploit the quasi-uniformity of the mesh (which gives $|\fbou |_{d-1}|\fbou' |_{d-1}\simeq h^{2d-2}$) to get
  \begin{align*}
    Q_1 \lesssim |\fbou |_{d-1} |\fbou' |_{d-1} 
    \sum_{\substack{ \fint \in \projFace^\dagger (\fbou' ) \\ \text{s.t.} \; \fbou  \subset \A_\fint \setminus \Delta_\fint}} \frac{h}{\delta_\fint^{d+1}}.  
  \end{align*} 
  We know from Lemma \ref{lem:Cases}-\ref{lem:Cases.1} that $\dffp \lesssim \delta_\fint$ for all $\fint$ in the sum above (for which $\projFace(\fint)=f'$). Furthermore, due to Remark \ref{scaling.dom}, we also have $\delta_\fint \lesssim 1$. Using this we decompose the sum over $\delta_\fint \in [lh, (l+1)h)$ for $l = l_0, \dots, L$ with $l_0$ and $L$ such that $l_0h \simeq \dffp$ and $L h \simeq 1$ (upon adjusting the hidden constant in the first $\simeq$ we can assume that $l_0\ge 2$). The cells $t \in \mathcal{T}_\fint$ such that $\fint \in\projFace^\dagger(\fbou' )$ and $\delta_\fint \in [lh,(l+1)h)$ are contained in an $h$-enlargement of the cylinder at the vertical of $\fbou'$ and between heights $lh$ and $(l+1)h$; the arguments in the proof of Lemma \ref{lem:Cardinality}-\ref{lem:Lm.cardinality} show that we have $\lesssim 1$ such cells, which means that we also have $\lesssim 1$ possible $\fint$, and thus
  $$
  \sum_{\substack{ \fint \in \projFace^\dagger (\fbou' ) \\ \text{s.t.} \; \fbou  \subset \A_\fint \setminus \Delta_\fint}} \frac{h}{\delta_\fint^{d+1}} = \sum_{l=l_0}^L \sum_{\substack{ \fint \in \projFace^\dagger (\fbou' ) \\ \text{s.t.} \; \fbou  \subset \A_\fint \setminus \Delta_\fint \\ \text{ and } \delta_\fint \in [lh, (l+1)h) }} \frac{h}{\delta_\fint^{d+1}} 
  \lesssim \sum_{l = l_0}^{L} \frac{h}{(lh)^{d+1}}.
  $$
  We estimate the last sum by writing
  \begin{equation*}
    \sum_{l=l_0}^{L} \frac{h}{(lh)^{d+1}} \leq \sum_{l=l_0}^{L} \int_{(l-1)h}^{lh} \frac{dx}{x^{d+1}} 
    = \int_{(l_0 - 1) h}^{L h} \frac{dx}{x^{d+1}} 
    \leq \frac{1}{d[(l_0-1)h]^d} 
    \leq {\frac1d}\left(\frac{l_0}{l_0-1}\right)^d \frac{1}{(l_0 h)^d} \lesssim \frac{1}{\dffp^d}. 
  \end{equation*}
  In conclusion, 
  $$
    Q_1 \lesssim   \frac{|f|_{d-1} |f'|_{d-1}}{\dffp^d}.
  $$ 	
  Lemma \ref{lem:At}-\ref{lem:At.4} and Lemma \ref{lem:Cases} lead to the following bound for the sum in $Q_2$:
  \begin{align*}
    Q_2 &\lesssim \sum_{\substack{\fint \in \projFace^\dagger (\fbou' ) \\ \text{s.t.} \; \fbou  \subset \Delta_\fint }} \frac{h^{d-1}}{\delta_\fint} \times \frac{h^{d-1}}{\delta_\fint^{d-1}} 
    \lesssim  \frac{|f|_{d-1} |f'|_{d-1}}{\dffp^d}.
  \end{align*}
  Putting the bounds for $Q_1$ and $Q_2$ in \eqref{eq:Qff} concludes the proof.
\end{proof}

\section{Extensions}\label{sec:extensions}

\subsection{Proof of the discrete trace and lifting inequalities in polytopal domains}\label{sec:trace.lift.polytopalcase} 

The idea to extend the proofs of the trace and lifting properties to the case of a generic polytope $\dom$ is as follows: (i) we select an atlas of a neighbourhood of $\partial \dom$ such that, in each chart, $\dom$ is the hypergraph of a Lipschitz function, (ii) with minor adaptations of the proofs made in Sections \ref{sec:proof.trace.cube} and \ref{sec:proof.lifting.cube} in the case of a cube, we show how the trace/lifting property can be obtained in each chart, and (iii) we glue together these local properties to obtain global trace/lifting properties.

This section is divided into three subsections. In Section \ref{sec:adaptation.local} we discuss the choice of the atlas and the adjustments to the definitions (in particular of the distances to the boundary and on the boundary) that need to be done, compared to the case of a cube, to establish trace and lifting properties in each chart.
The gluing of the local trace inequalities is then detailed in Section \ref{sec:polytopal.trace}, while that of the local liftings is addressed in Section \ref{sec:polytopal.lifting}.

\subsubsection{Localisation}\label{sec:adaptation.local}

Consider a generic polytope $\dom$. We can cover $\partial \dom$ by a finite family of open sets $(U_i)_{i \in I}$ such that for all $i \in I$,
\begin{align*}
  \dom \cap U_i= \left\{(x',x_d) \in V_i \times J_i: x_d > \phi_i(x') \right\},
\end{align*}
where $\phi_i: V_i \longrightarrow J_i$ is a Lipschitz function and $U_i=V_i\times J_i$ (see Figure \ref{fig:Polytopal-1-2} for an illustration). We note that, by regularity of $\dom$, we have $\# I\lesssim 1$ and that the Lipschitz constant of each $\phi_i$ is $\lesssim 1$.

\begin{figure}[htbp!]%
  \begin{tabular}{cc}
  \includegraphics[width=.3\linewidth]{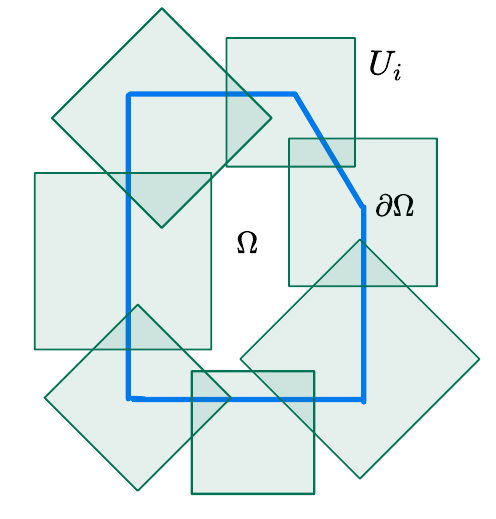} & 
  \includegraphics[width=.5\linewidth]{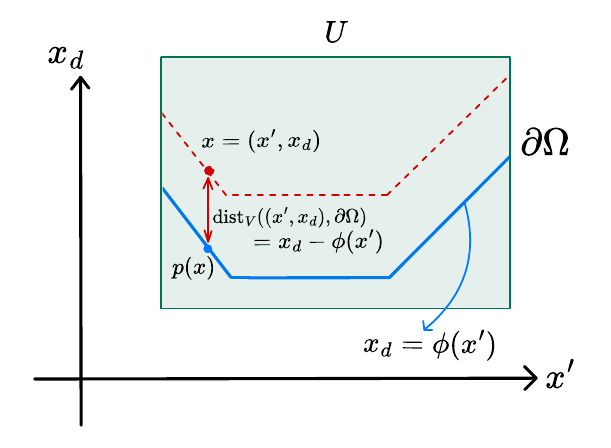}\\
  (A) & (B)
  \end{tabular}
  \caption{Illustration of localisation and chart on the boundary of a polytopal domain.}
  \label{fig:Polytopal-1-2}
\end{figure}

Let us describe the main adaptations to the construction made in Sections \ref{sec:proof.trace.cube}--\ref{sec:proof.lifting.cube} to establish the local trace and lifting property in each $\dom \cap U_i$. We momentarily drop the index $i$ as we work in a single map.

\medskip

\textbf{C1.} \ul{The notion of distances to/along $\partial \dom$}: In $\dom \cap U$, the distances to $\partial \dom$ (e.g., $\delta_t$) are expressed as \emph{vertical} distances (see Figure \ref{fig:Polytopal-1-2}-(B)), that is, for $x=(x',x_d)\in U$,
\begin{equation}\label{eq:dist.vertical.polytopal}
  \distV((x',x_d), \partial \dom) \doteq x_d - \phi(x').
\end{equation}
Accordingly, the projections $p(x_s)$ (for $s=t,g$) are also considered vertically in the coordinate system: if $x_t=(x'_t,x_{t,d})$ then $p(x_t)=(x'_t,\phi(x'_t))$.

Distances measured along $\partial \dom$ in the flat case (e.g., $\dist(p(x_t),f)$ in \eqref{def:At}) are now measured \emph{horizontally} on the $x'$-axis. For example, for $x=(x',x_d)\in\partial \dom\cap U$ and $y=(y',y_d)\in\partial \dom\cap U$,
\begin{equation}\label{eq:dist.dOmega.polytopal}
  \distH(x,y)\doteq |x'-y'|.
\end{equation}
Since $\partial \dom \cap U$ is the graph of a function, this formula defines a distance on that set.
As a consequence of this definition, if $f\in\Fhb$ is contained in $U$ and $x\in\partial\dom$, we set
\begin{equation}\label{eq:def.dist.dO}
  \distH(x,f)\doteq \min_{y\in f}\distH(x,y)=\min_{(y',y_d)\in f}|x'-y'|.
\end{equation}
With these notions of distances, for $t\in\Th$ such that $x_t\in U$ the local partition $\A_t$ is defined by \eqref{def:At} but restricting to the faces with centroids in $U$: 
$$
\A_t \doteq \bigcup \{ \mathrm{cl}(f): \text{$f \in \Fhb$, $x_f \in U$ and $\distH(p(x_t),f)\leq \delta_t$}\}.
$$
To show that this does not affect the results proved for $\A_t$ and related sets in Section \ref{sec:proof.technical.lifting} in the case of a cube, we show as an example that $\Delta_g$ still satisfies Lemma \ref{lem:At}-\ref{lem:At.2} in this polytopal case.

\begin{proposition}[Estimate on $\Delta_g$]
With the definitions above, we have $|\Delta_g|\lesssim h\delta_g^{d-2}$ for all $\fint\in\Fhi$.
\end{proposition}

\begin{proof}
Let $g \in \Fhi$ and $t,t' \in \mathcal{T}_g$. If $f \subset \A_t \backslash \A_{t'}$ then $\distH(x_f,p(x_t)) \leq \delta_t+h$ and $\distH(x_f,p(x_{t'})) > \delta_{t'}$. Due to Assumption \ref{assum:reg.mesh} we also have, since $t$ and $t'$ share a face and accounting for the definition \eqref{eq:dist.dOmega.polytopal} of distances along $\partial\dom\cap U$,
\begin{align}\label{eq:poly.eq1}
  2h \geq |x_{t} - x_{t'}| = \left(|x'_{t} - x'_{t'}|^2 + |x_{t,d} -x_{t',d}|^2\right)^{1/2} \geq |x'_t - x'_{t'}| = \distH(p(x_t),p(x_{t'})).
\end{align}
So
\begin{equation}\label{eq:poly.eq2}
  \delta_{t'}<\distH(x_f,p(x_{t'})) \leq \distH(x_f,p(x_t)) + \distH(p(x_t), p(x_{t'})) \leq \distH(x_f, p(x_t)) + 2h.
\end{equation}
Moreover, recalling that $\delta_{t'}$ is defined through the vertical distance \eqref{eq:dist.vertical.polytopal},
\begin{align}\label{eq:poly.eq3}
  \delta_{t'} = x_{t',d} - \phi(x'_{t'}) = (x_{t',d} - x_{t,d}) + \underbrace{(x_{t,d} - \phi(x'_t))}_{=\delta_t} + (\phi(x'_t) - \phi(x'_{t'})).
\end{align}
Denoting by $L_\phi$ the Lipschitz constant of $\phi$  we have $|\phi(x'_t) - \phi(x'_{t'})| \leq L_{\phi} |x'_{t} - x'_{t'}|$. Then, by \eqref{eq:poly.eq1},
$$
|x_{t',d} - x_{t,d}| + L_{\phi} |x'_t - x'_{t'}| \leq \sqrt{2} \max\{1,L_{\phi}\} \left(|x_{t',d} - x_{t,d}|^2 + |x'_t - x'_{t'}|^2\right)^{1/2} 
\leq 2\sqrt{2} \max\{1,L_{\phi}\} h.
$$
Substituting the above bounds in \eqref{eq:poly.eq3} gives $\delta_{t'} \geq \delta_t - 2\sqrt{2} \max\{1,L_{\phi}\} h$ which, plugged into \eqref{eq:poly.eq2}, leads to 
$$
\distH(x_f, p(x_t)) \geq \delta_t - 2(1+\sqrt{2} \max\{1,L_{\phi}\}) h.
$$
To summarise, for all $f \subset \A_t \backslash \A_{t'}$ we have $x_f \in \An{r_1}{r_2}{p(x_t)}$ (the annulus being defined along $\partial\dom\cap U$ based on the definition \eqref{eq:dist.dOmega.polytopal} of $\distH$ on that set) with $r_1 = \delta_t+h$, $r_2 = \delta_t - C_0 h$ and $C_0 = 2(1+\sqrt{2} \max\{1,L_{\phi}\})$.  By definition of the distance along the boundary, we have $\distH(x_f,y)\le h$ for all $y\in f$, so $f\subset \An{r_1+h}{r_2-h}{p(x_t)}$. Hence, $\A_t \backslash \A_{t'}$ is contained in this annulus, and
\begin{align*}
  | (\A_t \backslash \A_{t'}) |_{d-1}  \lesssim (\delta_t + h)^{d-2}\times(C_0 + 2)h \lesssim h \delta_t^{d-2}.
\end{align*}
Here the measure of the annulus is the $(d-1)$-dimensional measure along $\partial \dom\cap U$ -- for which  \eqref{MeasureOfAnnulus} is still valid (up to a multiplicative constant depending only on $\phi$) -- and we have used Lemma \ref{lem:At}--\ref{lem:At.0} to write $h\lesssim \delta_t$. Swapping the roles of $t$ and $t'$, using again Lemma \ref{lem:At}--\ref{lem:At.0} to get $\delta_t\simeq \delta_g$, and recalling that $\Delta_g = (\A_t \backslash \A_{t'}) \cup (\A_{t'} \backslash \A_{t})$ concludes the proof. 
\end{proof}

\medskip

\textbf{C2.} \ul{Layers $\La_m$}: The layers $\La_m$ are still given by \eqref{set:Layers} but with $\delta_t$ defined using the vertical distance \eqref{eq:dist.vertical.polytopal}, which leads to
\begin{align*}
  \La_m = \left\{t\in\Th: x_t=(x_t',x_{t,d})\in U\,,\;\phi(x_t')+mh\le x_{t,d}<\phi(x_t')+(m+1)h\right\}.
\end{align*} 
In other words, the layers are taken following the graph $\phi$ at a given height (see Figure \ref{fig:Polytopal-A2-A3}). 

\medskip

\textbf{C3.} \ul{Selecting cells}: Line segments along which we select cells (for example in the proof of the trace inequality) are either taken vertically above a fixed coordinate in $V$, as in the case of a cube, or following the graph of $\phi$ at a given height for lines that were ``horizontal'' in the cube (see, again, Figure \ref{fig:Polytopal-A2-A3}). With our definition of distance $\distV$ to the boundary, this ensures for example that all the cells $(t_j)_j$ appearing in \eqref{eq:bound.S2.eq1}, that are intersected by the ``line segment'' $\{(y',\phi(y')+ \distH(x_f,x_{f'}))\,:\,y'\in [x'_f,x'_{f'}]\}$, satisfy $\delta_{t_j}\simeq \distH(x_f,x_{f'})$.

\begin{figure}[htbp!]%
  \includegraphics[width=.5\linewidth]{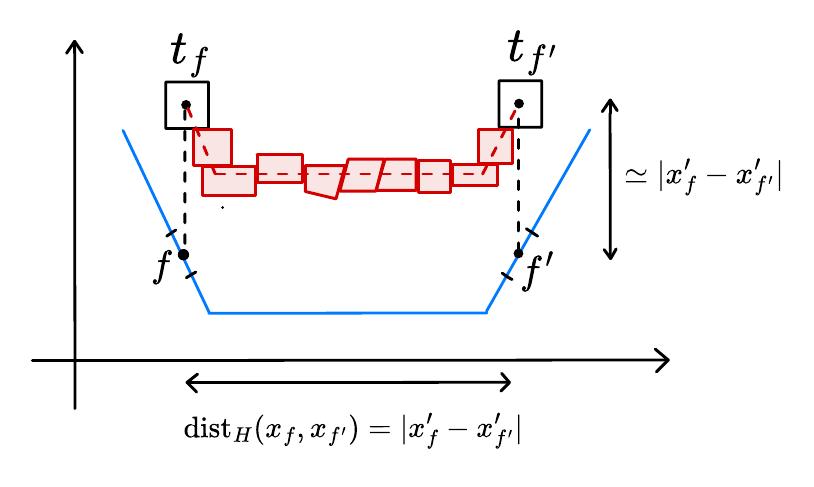}
  \caption{Local layer of cells following the boundary of a polytopal domain.}
  \label{fig:Polytopal-A2-A3}%
\end{figure}

We will now explain how local trace inequalities or lifting properties can be glued together to get Theorems \ref{thm:trace} and \ref{thm:lifting}.
From hereon, we re-introduce the index $i$ corresponding to the localisations.
If $X$ is a set of cells and $i\in I$, we denote by $X^i$ the set of cells in $X$ whose centroid belong to $U_i$; if $X$ is a set of faces, $X^i$ denotes the set of all faces in $X$ belonging to a cell whose centroid is in $U_i$. We extend this notation to the case where $X$ is a discrete space (of values on cells/faces): $X^i$ is the discrete space with values on cells (and/or faces belonging to these cells) whose centroids belong to $U_i$. 

It can be checked that there exists $h_0\gtrsim 1$ and $\zeta\gtrsim 1$ (depending only on the overlaps between the charts $(U_i)_{i\in I}$) such that, for all $h\le h_0$:
\begin{itemize}
\item For any $f\in\Fhb$, there exists $i\in I$ such that $f\subset \partial\dom\cap U_i$ (which shows in particular that \eqref{eq:def.dist.dO} makes sense for any face).
\item Pairs of nearby faces belong to the same local chart:
\begin{equation}\label{eq:choice.zeta}
\text{$\forall (f,f')\in \FFhb$ s.t.~$\dffp\le\zeta$, there exists $i\in I$ satisfying $(f,f')\in \FFhbloc$.}
\end{equation}
\item Non-empty overlaps $\partial\dom\cap (U_i\cap U_j)$ of boundary charts are sufficiently covered by faces with centroids in each chart, in such a way that
\begin{equation}\label{eq:h.small.enough}
\left|\bigcup (\Fhbloc{i}\cap \Fhbloc{j})\right|_{d-1}\ge \frac12 |\partial\dom\cap (U_i\cap U_j)|_{d-1}\gtrsim 1\quad\forall (i,j)\text{ such that }\partial\dom\cap (U_i\cap U_j)\not=\emptyset.
\end{equation}
\end{itemize}
There is no loss of generality in assuming that $h\le h_0$ since, if $h\ge h_0>0$, the trace inequality and lifting property trivially hold with constants depending only on $h_0$.

Finally, for $i\in I$, we define the following local discrete $H^1$- and $H^{1/2}$-seminorms, obtained by restricting the sums over the mesh entities whose centroid lie in $U_i$: for all $\ul{v}_h\in\Uh$,
\begin{equation}\label{eq:def.disc.H1.loc}
  \seminorm{1, h, i}{\ul{v}_h}^2 \doteq \left( \sum_{t\in\Thloc}\seminorm{1,t}{\ul{v}_t}^2\right)^{1/2}
\end{equation}
and, for all $w_h\in \Uhbb$,
\begin{equation}\label{eq:def.disc.Hhalf.loc}
  \tnorm{1/2, h, i}{w_h} \doteq \left(\sum_{f \in \Fhbloc{i}} h^{-1} \norm{L^2(f)}{w_f - \ol{w}_f}^2 + \sum_{(f,f') \in \FFhbloc} |f|_{d-1} |f'|_{d-1} \frac{|\ol{w}_f - \ol{w}_{f'}|^2}{\dffp^d}\right)^{1/2}.
\end{equation}

\subsubsection{Trace inequality}\label{sec:polytopal.trace}

With the adaptations described in Section \ref{sec:adaptation.local}, the arguments in Section \ref{sec:proof.trace} together with $\seminorm{1, h, i}{{\cdot}}\lesssim \seminorm{1, h}{{\cdot}}$ yield the following local trace inequalities: for all $\ul{v}_h\in\Uh$,
\[
\tnorm{1/2,h,i}{\tr(\ul{v}_h)}\lesssim \seminorm{1,h}{\ul{v}_h}\qquad\forall i\in I.
\]
The proof of the discrete trace inequality \eqref{eq:trace} is concluded by invoking the following proposition.

\begin{proposition}[Control of the discrete $H^{1/2}$-seminorm by local seminorms]
For all $w_h\in\Uhbb$ we have
\begin{equation}\label{eq:Hhalf.norm.from.local}
\left(\sum_{i\in I}\tnorm{1/2, h, i}{w_h}^2\right)^{1/2}\simeq \tnorm{1/2,h}{w_h}.
\end{equation}
\end{proposition}

\begin{proof}
The upper bound in \eqref{eq:Hhalf.norm.from.local} is trivial since $\#I\lesssim 1$ and $\tnorm{1/2, h, i}{{\cdot}}\le \tnorm{1/2, h}{{\cdot}}$. 
To get the lower bound, we first note that all the local contributions $h^{-1}\norm{L^2(f)}{w_f-\ol{w}_f}^2$ to $\tnorm{1/2, h}{w_h}^2$ appear in the left-hand side of \eqref{eq:Hhalf.norm.from.local}, since by choice of $h_0$ each $f\in\Fhb$ belongs to at least one $\Fhbloc{i}$.
It remains to bound the long-range contributions to $\tnorm{1/2,h}{w_h}^2$. 

Recalling the choice of $\zeta$ satisfying \eqref{eq:choice.zeta} and setting $\alpha_{f,f'}=|f|_{d-1}|f'|_{d-1}\frac{|\ol{w}_f-\ol{w}_{f'}|^2}{|x_f-x_{f'}|^d}$ to alleviate the notations, we can bound these contributions the following way:
\begin{align*}
\sum_{(f,f')\in\FFhb}\alpha_{f,f'} &= \sum_{(f,f')\in\FFhb,\,|x_f-x_{f'}|\le \zeta}\alpha_{f,f'} + \sum_{(f,f')\in\FFhb,\,|x_f-x_{f'}|> \zeta}\alpha_{f,f'}\\
\overset{\eqref{eq:choice.zeta}}&\le \sum_{i\in I}\sum_{(f,f')\in\FFhbloc}\alpha_{f,f'} 
+
\sum_{(i,j)\in I^2} \sum_{(f,f')\in\Fhbloc{i}\times\Fhbloc{j},\,|x_f-x_{f'}|> \zeta}\alpha_{f,f'}\\
&\lesssim \sum_{i\in I}\tnorm{1/2,h,i}{w_h}^2 + 
\frac{1}{\zeta^d}\sum_{(i,j)\in I^2} \underbrace{\sum_{(f,f')\in\Fhbloc{i}\times\Fhbloc{j}}|f|_{d-1}|f'|_{d-1}|\ol{w}_f-\ol{w}_{f'}|^2}_{\doteq  T_{ij}}.
\end{align*}
The conclusion follows if we prove that any $T_{ij}$ is bounded by the square of the left-hand side of \eqref{eq:Hhalf.norm.from.local}.

Let us first assume that $\partial\dom\cap (U_i\cap U_j)\not=\emptyset$. For any $F_{ij}\in\Fhbloc{i}\cap\Fhbloc{j}$, we have
$|\ol{w}_f-\ol{w}_{f'}|^2\le 2|\ol{w}_f-\ol{w}_{F_{ij}}|^2+2|\ol{w}_{F_{ij}}-\ol{w}_{f'}|^2$ and so
\begin{align*}
T_{ij}&\lesssim\sum_{(f,f')\in\Fhbloc{i}\times\Fhbloc{j}}|f|_{d-1}|f'|_{d-1}|\ol{w}_f-\ol{w}_{F_{ij}}|^2+\sum_{(f,f')\in\Fhbloc{i}\times\Fhbloc{j}}|f|_{d-1}|f'|_{d-1}|\ol{w}_{F_{ij}}-\ol{w}_{f'}|^2\\
&\lesssim \sum_{f\in\Fhbloc{i}}|f|_{d-1}|\ol{w}_f-\ol{w}_{F_{ij}}|^2+\sum_{f'\in\Fhbloc{j}}|f'|_{d-1}|\ol{w}_{F_{ij}}-\ol{w}_{f'}|^2,
\end{align*}
where, in the second inequality, we have used $\sum_{g\in\Fhbloc{k}}|g|_{d-1}\le |\partial\dom|_{d-1}\lesssim 1$ for $k=i,j$. Multiplying by $|F_{ij}|_{d-1}$ and summing over $F_{ij}\in \Fhbloc{i}\cap\Fhbloc{j}$ we infer
\begin{align*}
\left|\bigcup (\Fhbloc{i}\cap \Fhbloc{j})\right|_{d-1}T_{ij}\lesssim{}&
\sum_{(f,F_{ij})\in\Fhbloc{i}\times (\Fhbloc{i}\cap\Fhbloc{j})}|f|_{d-1}|F_{ij}|_{d-1}|\ol{w}_f-\ol{w}_{F_{ij}}|^2\\
&+\sum_{(F_{ij},f')\in(\Fhbloc{i}\cap\Fhbloc{j})\times\Fhbloc{j}}|f'|_{d-1}|F_{ij}|_{d-1}|\ol{w}_{F_{ij}}-\ol{w}_{f'}|^2\\
\lesssim{}& \tnorm{1/2,h,i}{w_h}^2+\tnorm{1/2,h,j}{w_h}^2,
\end{align*}
the conclusion following from $|x_g-x_{F_{ij}}|^d\le \mathrm{diam}(\partial\dom)^d$ for $g=f,f'$. Recalling \eqref{eq:h.small.enough} gives the desired bound on $T_{ij}$.

If $\partial\dom\cap (U_i\cap U_j)=\emptyset$, we can find a chain $U_i=U_{i(1)},U_{i(2)},\cdots,U_{i(n)}=U_j$ with $n\le \#I\lesssim 1$ such that
$\partial\dom\cap (U_{i(k)}\cap U_{i(k+1)})\not=\emptyset$ for all $k\in\{1,\ldots,n-1\}$. For all $k$ we take an arbitrary $F_{i(k)i(k+1)}\in \Fhbloc{i(k)}\cap \Fhbloc{i(k+1)}$ and write, for $(f,f')\in\Fhbloc{i}\times\Fhbloc{j}$,
\[
|w_f-w_{f'}|^2\overset{\eqref{CSinRn},\, n\lesssim 1}\lesssim |w_f-w_{F_{i(1)i(2)}}|^2+|w_{F_{i(1)i(2)}}-w_{F_{i(2)i(3)}}|^2+\cdots+|w_{F_{i(n-1)i(n)}}-w_{f'}|^2.
\]
We then reason in a similar way as in the intersecting case above: multiply this relation by the product $|f|_{d-1}|f'|_{d-1}|F_{i(1)i(2)}|_{d-1}\cdots |F_{i(n-1)i(n)}|_{d-1}$, sum over all the involved faces in their respective sets, and invoke \eqref{eq:h.small.enough} to see that $T_{ij}\lesssim \tnorm{1/2,h,i(1)}{w_h}^2+\tnorm{1/2,h,i(2)}{w_h}^2+\cdots+\tnorm{1/2,h,i(n)}{w_h}^2$.
\end{proof}

\subsubsection{Lifting property}\label{sec:polytopal.lifting} 

Given $w_h\in\Uhbb$ and following the same arguments as in Section \ref{sec:proof.lifting} -- with the adaptations mentioned in Section \ref{sec:adaptation.local} -- we can find for all $i\in I$ a local lifting $\ul{\tilde{v}}_{h,i} \in \Uhloc$ of the restriction $w_{h,i} \in \Uhbbloc$ of $w_h$, such that
\begin{align}\label{eq:lifting.local}
  \seminorm{1,h,i}{\ul{\tilde{v}}_{h,i}} \lesssim \tnorm{1/2, h, i}{w_{h,i}} \quad\mbox{and} \quad \tilde{v}_{f,i} = w_f \quad \forall  f \in \Fhbloc{i}.
\end{align}
Our aim here is to leverage these local liftings to create a global lifting $\ul{v}_h\in\Uh$ of $w_h \in \Uhbb$.

First, without loss of generality we can assume that
\begin{equation}\label{eq:wh.zero.integral}
  \int_{\partial \dom} w_h = 0.
\end{equation}
Indeed, this relation always holds upon adding a constant $C$ to $w_h$, and the constant function $C\in\Uhbb$ can be trivially lifted as the vector $((C)_{t\in\Th},(C)_{f\in\Fh})\in\Uh$.

Consider a partition of unity $(\eta_i)_{i \in I}$ on $\partial \dom$ associated with $(U_i)_{i \in I}$, such that
\begin{align*}
  \eta_i \in C_c^{\infty}(U_i), \quad 0 \leq \eta_i \leq 1, \quad \sum_{i \in I} \eta_i = 1 \mbox{ on $\partial \dom$}.
\end{align*}
Upon reducing $h_0$, we can also ensure that each $\eta_i$ vanishes on an $h_0$-neighbourhood of $\partial U_i$.
Recalling that $\lproj{X}{0}$ denotes the $L^2(X)$-orthogonal projection on $\mathbb{P}_0(X)$, and extending each $\ul{\tilde{v}}_{h,i}$ by zero on cells/faces that do not have their centroid in $U_i$, we then set $\ul{v}_{h,i} \in \Uh$ such that
\begin{equation}\label{eq:def.vhi}
  \begin{aligned}
    v_{t,i} ={}& (\lproj{t}{0} \eta_i) \tilde{v}_{t,i} \quad \forall t \in \Th, \\
    v_{f,i} ={}& (\lproj{f}{0} \eta_i) \tilde{v}_{f,i} \quad \forall f \in \Fh.
  \end{aligned}
\end{equation}

Define $\ul{v}_h = \sum_{i\in I} \ul{v}_{h,i}$. Using Proposition \ref{prop:lifting.reconstructed} below, we get
\begin{equation}\label{eq:lifting.ineq}
  \seminorm{1,h}{\ul{v}_h} \lesssim \tnorm{1/2, h}{w_h}.
\end{equation}
Furthermore, for all $f \in \Fhb$,
\begin{equation}\label{eq:lifting.trace}
  (\tr (\ul{v}_h))|_f = \sum_{i\in I} v_{f,i} = \sum_{i \in I} (\lproj{f}{0} \eta_i) \tilde{v}_{f,i} = \sum_{i \in I} (\lproj{f}{0} \eta_i) w_f = w_f,
\end{equation}
where the last two equalities are obtained by noticing that, since $\lproj{f}{0} \eta_i=0$ whenever $f\not\in\Fhbloc{i}$ (since $\eta_i$ vanishes on an $h$-neighbourhood of $\partial U_i$), 
\begin{align*}
  (\lproj{f}{0} \eta_i) \tilde{v}_{f,i} \overset{\eqref{eq:lifting.local}}= (\lproj{f}{0} \eta_i) w_f \;\; \mbox{and} \;\; \sum_{i \in I} (\lproj{f}{0} \eta_i) = \lproj{f}{0} (\sum_{i \in I} \eta_i) = \lproj{f}{0} 1= 1.
\end{align*}
The combination of \eqref{eq:lifting.ineq} and \eqref{eq:lifting.trace} concludes the proof of Theorem \ref{thm:lifting} in the case where $\dom$ is a polytope.

\medskip

The following lemma was used in the argument above.

\begin{proposition}\label{prop:lifting.reconstructed}
  Assuming that $w_h\in\Uhbb$ satisfies \eqref{eq:wh.zero.integral} and defining $\ul{v}_{h,i}$ by \eqref{eq:def.vhi} with $\ul{\tilde{v}}_{h,i}$ satisfying \eqref{eq:lifting.local}, we have
  $\seminorm{1,h}{\ul{v}_{h,i}} \lesssim \tnorm{1/2,h}{w_h}$.
\end{proposition}

The localisation process used in the proof of Proposition \ref{prop:lifting.reconstructed} generates $L^2$-norms of the discrete functions that have to be estimated in terms of their seminorms. To achieve this, we use the following two lemmas. The first one locally bounds the $L^2$-norm inside the domain by the discrete $H^1$-seminorm and the $L^2$-norm of the trace. The second one bounds the remaining $L^2$-norm of the trace by the discrete $H^{1/2}$-seminorm on the boundary.

\begin{lemma}\label{lem:gdm}
  For all $\ul{\widehat{v}}_{h,i} \in \bigtimes_{t\in\Th^i}\mathbb{P}_0(t)\times \bigtimes_{f\in\Fh^i}\mathbb{P}_0(f)$ (the lowest order hybrid space in $\dom\cap U_i$), it holds
  \begin{equation}\label{eq:L2.O.Ui}
  \norm{L^2(\dom \cap U_i)}{\widehat{v}_{h,i}} \lesssim \seminorm{\dom \cap U_i,2}{\ul{\widehat{v}}_{h,i}} + \norm{L^2(\partial \dom \cap U_i)}{\tr (\ul{\widehat{v}}_{h,i})},
  \end{equation}
  where $\widehat{v}_{h,i}$ is the piecewise constant function defined by $(\widehat{v}_{h,i})_{|t}=\widehat{v}_t$ for all $t\in\Th^i$, and 
  $$
    \seminorm{\dom \cap U_i,2}{\ul{\widehat{v}}_{h,i}}^2 \doteq \sum_{t \in \Th^i} \sum_{f \in \mathcal F_t} \frac{|f|_{d-1}}{d_{t,f}} |\widehat{v}_{f,i} - \widehat{v}_{t,i}|^2,
  $$
  with $d_{t,f}$ denoting the orthogonal distance between $f$ and the centre of the largest ball contained in $t$.
\end{lemma}

\begin{proof}
  \cite[Lemma B.22]{Droniou.Eymard:2018:GDM} establishes the global version of \eqref{eq:L2.O.Ui}, that is, when $U_i$ is not present. However, the proof in this reference provides a more precise result. Specifically, it consists in selecting a unit vector $\mathbf{e}$ and in bounding the value of a hybrid function on a cell $t$ using (i) its local discrete $H^1$-seminorms on cells crossed by $t+\mathbb R^+ \mathbf{e}$, and (ii) its boundary values on $(t+\mathbb R^+ \mathbf{e})\cap \partial \dom$. Given the structure of our localisation, if we take $\mathbf{e}=(0,\ldots,0,-1)$ in the coordinates system corresponding to $U_i$ and apply the argument in the proof of \cite[Lemma B.22]{Droniou.Eymard:2018:GDM}, we naturally get \eqref{eq:L2.O.Ui} by noticing that, for $t\subset\dom\cap U_i$, (i) the cells crossed by $t+\mathbb R^+ \mathbf{e}$ are in $\Th^i$, and (ii) $(t+\mathbb R^+ \mathbf{e})\cap \partial \dom\subset U_i\cap\partial\dom$.
  
  To be completely rigorous, $t+\mathbb R^+ \mathbf{e}$ could cross some cells $s$ that are not in $\Th^i$ -- when $s$ overlaps with $U_i$ but $x_s\not\in U_i$. This situation can be fixed by defining $\widehat{v}_{s,i}$ as the average of the face values $\widehat{v}_{f,i}$ for $f\in\Fh^i\cap\mathcal F_s$, by fixing $\widehat{v}_{g,i}=\widehat{v}_{s,i}$ for all $g\in\mathcal F_s\backslash\Fh^i$, and by noticing that the added term $\sum_{f\in\mathcal F_s}\frac{|f|_{d-1}}{d_{s,f}}|\widehat{v}_{f,i}-\widehat{v}_{s,i}|^2$ in $\seminorm{\Omega\cap U_i,2}{\ul{\widehat{v}}_{h,i}}^2$ can be bounded by the other terms already present in this seminorm, since all $f\in\Fh^i\cap\mathcal F_s$ are shared with cells in $\Th^i$.
\end{proof}

\begin{lemma}\label{lem:PW}
  Let $P_0\subset \partial \dom$ be the union of certain faces in $\Fhb$. For all $w_h \in \Uhbb$ such that $\int_{P_0} w_h  = 0$, we have 
  $$
  \norm{L^2(\partial \dom)}{w_h} \lesssim \left(1 + \frac{\mathrm{diam}(\partial \dom)^d}{|P_0|_{d-1}}\right)^{1/2}\tnorm{1/2, h}{w_h}.
  $$
\end{lemma} 

\begin{proof} We have 
\begin{align}
  \norm{L^2(\partial \dom)}{w_h}^2 &= \sum_{f \in \Fhb} \norm{L^2(f)}{w_f}^2 \nonumber\\
  &\leq 2 \sum_{f \in \Fhb} \norm{L^2(f)}{w_f - \ol{w}_f}^2 + 2 \sum_{f \in \Fhb} \norm{L^2(f)}{\ol{w}_f}^2 \nonumber\\
  \overset{h_f\lesssim 1}&\lesssim \sum_{f \in \Fhb} h_f^{-1} \norm{L^2(f)}{w_f - \ol{w}_f}^2 +  \sum_{f \in \Fhb} |f|_{d-1} |\ol{w}_f|^2 \nonumber\\
  &\lesssim \tnorm{1/2, h}{w_h}^2 + \sum_{f \in \Fhb} |f|_{d-1} |\ol{w}_f|^2.
\label{eq:PW.eq1}
\end{align}
Let us now work on the second term in \eqref{eq:PW.eq1}. Notice that
\begin{align*}
  0 = \int_{P_0} w_h  = \sum_{f' \in \Fhb(P_0)} \int_{f'} w_{f'} = \sum_{f' \in \Fhb(P_0)} |f'|_{d-1} \ol{w}_{f'},
\end{align*}
where $\Fhb(P_0)$ is the set of faces that form $P_0$.
Using this relation, we can write
\begin{align*}
  \sum_{f \in \Fhb} |f|_{d-1} |\ol{w}_f|^2 &= \sum_{f \in \Fhb} |f|_{d-1} \left|\ol{w}_f - \frac{1}{|P_0|_{d-1}} \sum_{f' \in \Fhb(P_0)}|f'|_{d-1} \ol{w}_{f'} \right|^2 \\
  &= \sum_{f \in \Fhb} |f|_{d-1} \left|\frac{1}{|P_0|_{d-1}} \sum_{f' \in \Fhb(P_0)} |f'|_{d-1} (\ol{w}_f - \ol{w}_{f'})\right|^2\\
  &\le \sum_{f \in \Fhb} |f|_{d-1} \frac{1}{|P_0|_{d-1}} \sum_{f' \in \Fhb(P_0)} |f'|_{d-1} |\ol{w}_f - \ol{w}_{f'}|^2,
\end{align*} 
where the last inequality comes from the convexity of $s\to s^2$. For all 
$(f,f') \in \FFhb$ we have $\dffp \lesssim \mathrm{diam}(\partial \dom)$. Hence,
\begin{align*}
  \sum_{f \in \Fhb} |f|_{d-1} |\ol{w}_f|^2 \lesssim \frac{\mathrm{diam}(\partial \dom)^d}{|P_0|_{d-1}} \sum_{(f,f') \in \FFhb} |f|_{d-1} |f'|_{d-1} \frac{|\ol{w}_f - \ol{w}_{f'}|^2}{\dffp^d} \leq \frac{\mathrm{diam}(\partial \dom)^d}{|P_0|_{d-1}} \tnorm{1/2, h}{w_h}^2.
\end{align*} 
Plugging the above inequality into \eqref{eq:PW.eq1} concludes the proof.
\end{proof}

\medskip
We are now ready to prove Proposition \ref{prop:lifting.reconstructed}.

\begin{proof}[Proof of Proposition \ref{prop:lifting.reconstructed}]
  For $t \in \Th$ and $f \in \mathcal F_t$, we have
$$
  v_{t,i} - v_{f,i} \overset{\eqref{eq:def.vhi}}=  (\lproj{t}{0} \eta_i) \tilde{v}_{t,i} -  (\lproj{f}{0} \eta_i) \tilde{v}_{f,i} 
  = (\lproj{t}{0} \eta_i - \lproj{f}{0} \eta_i) \tilde{v}_{t,i} + (\lproj{f}{0} \eta_i)(\tilde{v}_{t,i} - \tilde{v}_{f,i}).
$$ 
  Since $\eta_i$ is smooth, $|\lproj{t}{0} \eta_i - \lproj{f}{0} \eta_i|\lesssim h_t$ and $|\eta_i|\lesssim 1$, so
  \begin{align*}
    h_t^{-1} \norm{L^2(f)}{v_{t,i} - v_{f,i}}^2 \lesssim h_t \norm{L^2(f)}{\tilde{v}_{t,i}}^2 + h_t^{-1} \norm{L^2(f)}{\tilde{v}_{t,i} - \tilde{v}_{f,i}}^2. 
  \end{align*} 
  Noticing that $v_{t,i}$ is constant (recall \eqref{eq:reconstructed.vt} and \eqref{eq:def.vhi}), the definition \eqref{eq:def.disc.H1} of $\seminorm{1,t}{{\cdot}}$ and the discrete trace inequality \cite[Lemma 1.32]{di-pietro.droniou:2020:hybrid} then give
  \begin{equation}\label{eq:lift.bound.local}
    \seminorm{1,t}{\ul{v}_{t,i}}^2 \lesssim \norm{L^2(t)}{\tilde{v}_{t,i}}^2 + \seminorm{1,t}{\ul{\tilde{v}}_{h,i}}^2. 
  \end{equation} 
  By \eqref{eq:def.vhi} and choice of the support of $\eta_i$, we have $\ul{v}_{t,i}=\ul{0}$ whenever $t\not\subset U_i$, from which we infer
  \begin{equation}\label{eq:lifting.reconstructed.eq0}
    \seminorm{1,h}{\ul{v}_{h,i}}^2 =\sum_{t\subset U_i}\seminorm{1,t}{\ul{v}_{t,i}}^2\overset{\eqref{eq:lift.bound.local}}\lesssim \norm{L^2(\dom \cap U_i)}{\tilde{v}_{h,i}}^2 + \seminorm{1, h,i}{\ul{\tilde{v}}_{h,i}}^2 \lesssim  \norm{L^2(\dom \cap U_i)}{\tilde{v}_{h,i}}^2 + \tnorm{1/2, h}{w_h}^2,
  \end{equation} 
  where we have used \eqref{eq:lifting.local} and $\tnorm{1/2, h,i}{w_{h,i}}\le \tnorm{1/2, h}{w_h}$ in the last inequality.
  
  We deal with the first term in the right-hand side of \eqref{eq:lifting.reconstructed.eq0} using Lemma \ref{lem:gdm}. Let $\ul{\widehat{v}}_{h,i}$ be the lowest-order projection of $\ul{\tilde{v}}_{h,i}$, that is: $\widehat{v}_{t,i}=\tilde{v}_{t,i}\in\Poly{0}(t)$ for all $t\in\Th$ (recall that the lifting $\ul{\tilde{v}}_{h,i}$ has constant cell values) and $\widehat{v}_{f,i}=\lproj{f}{0}\tilde{v}_{f,i}\in\Poly{0}(f)$ for all $f\in\Fh$. Then, Lemma \ref{lem:gdm} gives
   \begin{equation}\label{eq:lifting.reconstructed.eq1}
    \norm{L^2(\dom \cap U_i)}{\tilde{v}_{h,i}}=\norm{L^2(\dom \cap U_i)}{\widehat{v}_{h,i}} \lesssim \seminorm{\dom \cap U_i,2}{\ul{\widehat{v}}_{h,i}} + \norm{L^2(\partial \dom \cap U_i)}{\tr (\ul{\widehat{v}}_{h,i})}. 
  \end{equation}
  Using $d_{t,f} \simeq h_t$ (by mesh regularity) and applying \eqref{lem:local-tf-approximation.b} to $\ul{\tilde{v}}_{t,i}$, we get $\seminorm{\dom \cap U_i,2}{\ul{\widehat{v}}_{h,i}} \lesssim \seminorm{1,h,i}{{\ul{\tilde{v}}_{h,i}}}$. Moreover, for all $f\in\Fhb$, $\norm{L^2(f)}{\widehat{v}_{f,i}}
  =\norm{L^2(f)}{\lproj{f}{0}\tilde{v}_{f,i}}\le \norm{L^2(f)}{\tilde{v}_{f,i}}\overset{\eqref{eq:lifting.local}}=\norm{L^2(f)}{w_f}$. Plugging these estimates into \eqref{eq:lifting.reconstructed.eq1} yields
   \begin{equation}\label{eq:lifting.reconstructed.eq2}
    \norm{L^2(\dom \cap U_i)}{\tilde{v}_{h,i}} \lesssim \seminorm{1,h,i}{{\ul{\tilde{v}}_{h,i}}}+ \norm{L^2(\partial \dom)}{w_h}
    \overset{\eqref{eq:lifting.local}}\lesssim  \seminorm{1/2,h}{w_h},
  \end{equation}
  where, in the conclusion, we have additionally used Lemma \ref{lem:PW} with $P_0 = \partial \dom$, recalling that $w_h$ satisfies \eqref{eq:wh.zero.integral}.
  Plugging \eqref{eq:lifting.reconstructed.eq2} into \eqref{eq:lifting.reconstructed.eq0} concludes the proof of Proposition \ref{prop:lifting.reconstructed}.
\end{proof}

\subsection{More general hybrid spaces}\label{sec:general.hybrid.space}

A range of polytopal methods are based on discrete spaces spanned by vectors of polynomials in the cells and on the faces: Hybrid High-Order \cite{di-pietro.droniou:2020:hybrid}, Hybridizable Discontinuous Galerkin \cite{cockburn.dong.ea:2009:hybridizable}, non-conforming Virtual Element Method \cite{Dios-Lipnikov-Manzini-ncVEM}, Weak Galerkin \cite{mu.wang.ea:2015:weak}, etc. On the contrary to those in $\Uh$, these polynomials may not be of equal degree in the cells and on the faces, or even not of the same degrees on all cells/faces. Nonetheless, an easy argument shows that the results of Theorems \ref{thm:trace} and \ref{thm:lifting} still apply to these situations.

Let us briefly detail this argument when the underlying space $\underline{V}_h$ of a considered method is made of polynomials having different degrees in cells and faces:
for some $\ell,r\ge 0$,
$$
  \underline{V}_h=\bigtimes_{t\in\Th}\mathbb{P}_\ell(t)\times \bigtimes_{f\in\Fh}\mathbb{P}_r(f).
$$
Setting $k=\max(\ell,r)$, we have $\underline{V}_h\subset \Uh$. The restriction to $\underline{V}_h$ of the trace operator $\gamma:\Uh\to\Uhbb$ yields a trace operator $\gamma_{V,h}:\underline{V}_h\to V_h^{\partial, \boundary}$ where $V_h^{\partial, \boundary}\doteq\bigtimes_{f\in\Fhb}\mathbb{P}_r(f)$.
Endowing $\underline{V}_h$ and $V_h^{\partial,\boundary}$ with the restrictions of the discrete norms \eqref{eq:def.disc.H1} and \eqref{eq:def.disc.Hhalf}, respectively, Theorem \ref{thm:trace} applied to $\Uh$ directly gives a trace inequality for $\gamma_{V,h}$. In practical situations, it can be checked that the energy norm of a specific method is equivalent to the restriction of \eqref{eq:def.disc.H1} on its space; hence, the trace inequality obtained for $\gamma_{V,h}$ is a suitable one, as it can be written in terms of this energy norm on $\ul{V}_h$.

Concerning the lifting, since $V_h^{\partial,\boundary}\subset \Uhbb$, using Theorem \ref{thm:lifting} we can lift any $w_h\in V_h^{\partial, \boundary}$ into $\ul{v}_h\in\Uh$. An examination of the construction of this lifting (see, in particular, \eqref{eq:reconstructed} and \eqref{eq:def.vhi}) shows that $v_t\in\mathbb{P}_0(t)$ for all $t\in\Th$ and $v_f\in\mathbb{P}_0(f)$ for all $f\in\Fhi$. Since $v_f=w_f\in\mathbb{P}_r(f)$ for all $f\in\Fhb$, we infer that the lifting $\ul{v}_h$ actually belongs to $\underline{V}_h$. This argument therefore gives a lifting $V_h^{\partial, \boundary}\to\underline{V}_h$, with suitable seminorm estimates.

\begin{remark}[No cell values]
Some low-order schemes (e.g., some variants of HHO \cite[Section 5.1]{di-pietro.droniou:2020:hybrid}) do not have cell values. In this case, the lifted $\ul{v}_h$ constructed above does not belong to the space of the method. An easy work-around then consists in simply discarding the cell unknowns in $\ul{v}_h$ and in only keeping the face unknowns.
\end{remark}

\begin{remark}[No face values]
Some polytopal methods rely on discrete spaces without face values -- for example, broken polynomial spaces on the mesh cells as in DG.
These spaces can however easily be embedded into $\Uh$, by adding face values equal to the averages of the traces of the two cell values on each side of each face (or just the trace of the cell value containing a boundary face). This allows for a straightforward adaptation of Theorems \ref{thm:trace} and \ref{thm:lifting} to these methods.
\end{remark}

%%%%%%%%%%%%%%%%%%%%%%%%%%%%%%%%%%%%%%%%%%%%%%%%%%%%%%%%%%%%%%%%%%%%%%%%%%%%%%%%%%%%%%%%%%%%%%%%%%%%%%%%%%%%%%%%%%%%%%%%%

\section{Numerical tests}\label{sec:NumExp}

In this section, we present numerical simulations that validate the theoretical results. The experiment is conducted in \textit{Gridap} \cite{Badia2020,Verdugo-Badia-2022-Gridap}, an open source PDE approximation toolbox written in the Julia programming language. The domain is a unit square meshed with a family of uniform Cartesian meshes that satisfy the mesh regularity Assumption \ref{assum:reg.mesh}. The test verifies the equivalence, for discretely harmonic functions, between the discrete $H^1$-seminorm on the domain and the discrete $H^{1/2}$-seminorm on the boundary, which is a consequence of Theorems \ref{thm:trace} and \ref{thm:lifting}. 
We present the construction of the test, followed by a brief explanation of the results. 

% \subsection{Trace and lifting estimates}

Given $w_h \in \Uhbb$, let us define the discrete harmonic extension $\mathcal{E}_h({w}_h) \in \Uh$ such that $\mathcal{E}_h({w}_h) = {w}_h$ on $\partial\dom$ and $\mathcal E_h(w_h)$ minimises the $H^1(\dom)$-seminorm; in other words, $\mathcal E_h(w_h)$ solves 
$$
  \text{$\mathcal{E}_h(w_h) = w_h$ on $\partial\dom$, and $\langle \mathcal{A}_h \mathcal{E}_h(w_h), \ul{v}_h \rangle = 0$ for all $\ul{v}_h \in \ul{U}_{h,0}$},
$$ 
where $\mathcal{A}_h$ is the operator corresponding to the discrete $H^1(\dom)$-seminorm and $\ul{U}_{h,0}$ is the subspace of $\Uh$ with zero trace on $\partial\dom$. The following seminorm equivalence is a direct consequence of Theorems \ref{thm:trace} and \ref{thm:lifting}, and of the definition of the discrete harmonic extension:
\begin{align}\label{eq:equivalence}
\seminorm{1,h}{\mathcal{E}_h(w_h)} \simeq \tnorm{1/2, h}{{w}_h}  \qquad \forall w_h \in \Uhbb.
\end{align}
Using Lemma \ref{lem:PW}, we note that this seminorm equivalence is actually a norm equivalence on the space $\Uhbb \cap L_0^2(\partial\dom)$, where $L_0^2(\partial\dom)$ is the subspace of $L^2(\partial\dom)$ spanned by functions with zero integral on the boundary.

We consider a suitable choice of basis vectors spanning the hybrid space $\Uh$, following an approach analogous to the one detailed in \cite[Appendix B]{di-pietro.droniou:2020:hybrid}. We number the boundary faces as $(f_m)_{m = 1, \ldots, \# \Fhb}$, and label shape functions $\phi_i^{f_m} \in \Uhbb$ on each face $f_m$ by a face-local index $i \in \{1, \ldots, \mathcal{N}_{f_m}\}$, $\mathcal{N}_{f_m}{=\mathrm{dim}(\mathbb{P}_k(f_m))}$ being the number of degrees of freedom belonging to $f_m$. 

Let $\HoneOp$ and $\HhalfOp$ be the Gram matrices associated to the discrete $H^1(\dom)$- and $H^{1/2}(\partial\Omega)$-seminorms \eqref{eq:def.disc.H1} and \eqref{eq:def.disc.Hhalf}, respectively. $\HoneOp$ has the following block structure:
\[
\HoneOp = 
\renewcommand\arraystretch{1}
\mleft[
\begin{array}{c c c}
\HoneOp_{\Th \Th} & \HoneOp_{\Th \Fhi} & \HoneOp_{\Th \Fhb} \\
\HoneOp_{\Fhi \Th} & \HoneOp_{\Fhi \Fhi} & \HoneOp_{\Fhi \Fhb} \\
\HoneOp_{\Fhb \Th} & \HoneOp_{\Fhb \Fhi} & \HoneOp_{\Fhb \Fhb} 
\end{array}
\mright].
\]
We can compute the Schur complement of this matrix with respect to boundary degrees of freedom, denoted by $\HoneOp_{SC}$, as
$$
\HoneOp_{SC} \doteq \HoneOp_{\Fhb \Fhb} - \left[ \HoneOp_{\Fhb \Th} \;\; \HoneOp_{\Fhb \Fhi} \right] \begin{bmatrix}
\HoneOp_{\Th \Th} & \HoneOp_{\Th \Fhi} \\
\HoneOp_{\Fhi \Th} & \HoneOp_{\Fhi \Fhi}
\end{bmatrix}^{-1} 
\begin{bmatrix}
\HoneOp_{\Th \Fhb} \\
\HoneOp_{\Fhi \Fhb}
\end{bmatrix}.
$$
$\HoneOp_{SC}$ is a symmetric and positive semi-definite matrix of dimension $\mathcal{N}^{\partial, \boundary} \times \mathcal{N}^{\partial, \boundary}$, where $\mathcal{N}^{\partial, \boundary}$ is the dimension of $\Uhbb$. 

In order to build the hybrid operator $\HhalfOp$, we restrict the shape functions to the ones associated with the boundary faces, that is, {the shape functions that span} $\Uhbb$. 
The non-local part of $\HhalfOp$ can be written in matrix form by assembling contributions by pairs of faces.

Given a pair of boundary faces $(f_l,f_m) \in \Fhb \times \Fhb$, we define the corresponding local contribution $K \boldsymbol{\Delta}^{\top} \boldsymbol{\Delta}$ as:
\begin{align*}
&K \doteq  \frac{|f_l|_{d-1} |f_m|_{d-1}}{|x_{f_l} - x_{f_m}|^d} 
, \qquad
&\boldsymbol{\Delta}_{ij} \doteq \left[ \bar{\phi}^{f_l}_i - \bar{\phi}^{f_m}_j \right]_{1 \leq i \leq \mathcal{N}_{f_l}, 1 \leq j \leq \mathcal{N}_{f_m} }. 
\end{align*}
This local 4D array is assembled into the global matrix $\HhalfOp$ by using a local-to-global indexing of the degrees of freedom associated with the boundary faces. The local part of the operator $\HhalfOp$ can be computed analogously to $\HoneOp$, resulting in a matrix $\HhalfOp$ of dimension $\mathcal{N}^{\partial, \boundary} \times \mathcal{N}^{\partial, \boundary}$. 

To avoid zero eigenvalues and indefinite matrices, we proceed as follows. We define the rank-one matrix $\boldsymbol{S}^{\top} \boldsymbol{S}$, where $\boldsymbol{S}\doteq [ \int_{\partial \dom} \phi_a ]_{1 \leq a \leq \mathcal{N}^{\partial, \boundary}}$ is a row vector, $\phi_a$ denoting the shape functions of $\Uhbb$ labelled with global index $a$. 
Thus, if $\boldsymbol{w}$ is the vector representing $w_h\in\Uhbb$ in the chosen basis on $\Uhbb$, we have $\boldsymbol{S} \boldsymbol{w} = \int_{\partial \Omega} {w}_h$ and $\boldsymbol{w}^{\top} \boldsymbol{S}^{\top} \boldsymbol{S} \boldsymbol{w} = (\int_{\partial \Omega} w_h)^2$; $\boldsymbol{S}$ thus serves the purpose of fixing the kernels of the operators $\HhalfOp$ and $\HoneOp_{SC}$. Since the bilinear form $\boldsymbol{S}^{\top} \boldsymbol{S}$ vanishes on vectors corresponding to functions with zero integral, the seminorm equivalence \eqref{eq:equivalence} readily implies:
\begin{align}\label{eq:equivalence2}
\boldsymbol{w}^{\top} \left( \HhalfOp + \boldsymbol{S}^{\top} \boldsymbol{S} \right) \boldsymbol{w} \simeq \boldsymbol{w}^{\top} \left( \HoneOp_{SC} + \boldsymbol{S}^{\top} \boldsymbol{S} \right) \boldsymbol{w}, \qquad \forall {w}_h \in \Uhbb \cap L^2_0(\partial \dom).
\end{align}
Using the fact that both sides in the above equality match for vectors corresponding to constant functions, and that $(\Uhbb \cap L^2_0(\partial \dom)) \oplus \mathbb{R} = \Uhbb$, we see that \eqref{eq:equivalence2} implies the equivalence \eqref{eq:equivalence} for all functions in $\Uhbb$. 

We experimentally verify the equivalence in \eqref{eq:equivalence2} by comparing the eigenvalues of the operators $\HhalfOp$ and $\HoneOp_{SC}$. In particular, we solve the following generalised eigenvalue problem:
\begin{align}\label{eq:EVP}
\left[ \left( \HhalfOp + \boldsymbol{S}^{\top} \boldsymbol{S} \right)^{-1} \left( \HoneOp_{SC} + \boldsymbol{S}^{\top} \boldsymbol{S} \right) \right] \boldsymbol{W} = \boldsymbol{W} \boldsymbol{\Lambda}.
\end{align}
Here, the diagonal elements $\boldsymbol{\Lambda} \doteq \boldsymbol{\mbox{diag}}(\lambda_i)_{i=1,\ldots,\mathcal{N}^{\partial, \boundary}}$ are the eigenvalues and the columns of $\boldsymbol{W}$ are the eigenvectors. In order for \eqref{eq:equivalence2} to hold, the eigenvalues of \eqref{eq:EVP} should be bounded away from zero and bounded above by a constant independent of the mesh size.

Figure \ref{fig:EigenvaluesVsDofs} presents the maximum and minimum eigenvalues of \eqref{eq:EVP} against the number of degrees of freedom associated with the hybrid space $\Uh$, for different polynomial orders. 
We see an asymptotic trend towards a constant for the maximum and minimum eigenvalues as we increase the number of degrees of freedom, indicating that the eigenvalues of \eqref{eq:EVP} are indeed bounded above and below independently of $h$, which verifies the correctness of Theorems \ref{thm:trace} and \ref{thm:lifting}. 
\begin{figure}[!htbp]
  \centering
  \begin{subfigure}[b]{0.48\linewidth}
    \centering
    \includegraphics[width=\linewidth]{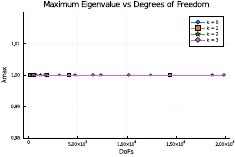}
    \caption{Maximum eigenvalues}
    \label{fig:MaxEigenvalues}
  \end{subfigure}
  \hfill
  \begin{subfigure}[b]{0.48\linewidth}
    \centering
    \includegraphics[width=\linewidth]{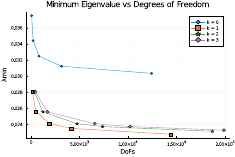}
    \caption{Minimum eigenvalues}
    \label{fig:MinEigenvalues}
  \end{subfigure}
  \caption{Eigenvalues of \eqref{eq:EVP} versus the number of degrees of freedom of $\Uh$ for polynomial orders $0, 1, 2$ and $3$.}
  \label{fig:EigenvaluesVsDofs}
\end{figure}

\section{Appendix: Tables of notations}\label{sec:appendix}

We collect in Tables \ref{Table:Notation_Symbols} and \ref{Table:Notation_Sets} the notations used in the paper. The first table contains symbols related to points, faces and distances, while the second describes symbols related to sets of faces and cells, as well as other relevant notations.

\begin{table}[h!]
\centering
 \renewcommand*{\arraystretch}{1.2}
  \begin{tabular}{| c | l |}
      \hline
      \textbf{Symbol} & \textbf{Description} \\
      \hline
      $p(x)$ & Cubic $\dom$: orthogonal projection of $x\in \dom$ on a face of $\dom$ (see Figure \ref{fig:Slicing-6-7}-(B)). \\
              & Polytopal $\dom$: in a local chart, vertical projection of $x\in\dom$ on $\partial\dom$ (see Figure \ref{fig:Polytopal-1-2}-(B)).\\
      \hline
      $x_Z$  & Centroid of $Z$, a face or cell of the mesh. \\
      \hline
      $D(x,r)$ & Disc on $\partial\dom$ centered at $x$ with radius $r$. \\
      \hline
      $\delta_Z$& Distance between $x_Z$ and $\partial\dom$ (for $Z$ a mesh cell or face), defined by $|x_Z - p(x_Z)|$. \\
      \hline
      $\dist(z,f)$ & Distance between $z \in \partial\dom$ and $f \subset \partial\dom$, when $\partial\dom$ is flat.\\
      \hline
      $\distV(x,\partial\dom)$ & In a local chart, vertical distance from $x\in\dom$ to $\partial\dom$. See \eqref{eq:dist.vertical.polytopal}.\\
      \hline
      $\distH(x,y)$ & In a local chart, horizontal distance between $x,y\in\partial\dom$. See \eqref{eq:dist.dOmega.polytopal}.\\
      \hline
  \end{tabular}
  \caption{Points, faces and distance}
  \label{Table:Notation_Symbols}
\end{table}

\begin{table}[h!]
\centering
 \renewcommand*{\arraystretch}{1.2}
  \begin{tabular}{| c | l |}
        \hline
      \textbf{Symbol} & \textbf{Description} \\
      \hline
      $\FFhb$ & Set of pairs $(f,f')$ of distinct faces on $\partial\Omega$. See \eqref{eq:def.FFhb}.\\
      \hline
      $\An{r_1}{r_2}{x}$ & Annulus on $\partial \dom$ centered at $x$ and with radii $r_1$ and $r_2$ such that $r_1 > r_2$. \\
      \hline
      $\W_l$ & Set of pairs of boundary faces that are \emph{approximately at distance $\simeq lh$} \\ & \emph{of each other}. See \eqref{Wl}. \\
      \hline
      $\W_{lf}$ & Slice of the set $\W_l$ at the face $f$, that is, faces $f'$ that 
                are \emph{approximately} \\ & \emph{at distance $lh$ of $f$}. See \eqref{Rfl}. \\
      \hline
      $\Above_{\fbou}$ & Set of cells \emph{at the vertical of $f$}. See \eqref{def:set.above}. \\
      \hline
      $\La_m$ & Layer of cells in $\Th$ that are at distance $\simeq mh$ of $\partial\dom$. See \eqref{set:Layers}. \\
      \hline
      $\Iffp$ & Set of cells between $f,f'\in\Fhb\times\Fhb$ at height $|x_{f} - x_{f'}|$ above $\partial\dom$. \\  & See \eqref{set:Iffp} and Figure \ref{fig:Slicing-6-7}-(A). \\
      \hline
      $\Ca$ & Set of cells in $\Iffp$ whose projection of the centroid on $\partial\dom$ is at distance $\simeq sh$ of $f$.\\& See \eqref{set:HorizontalLayers} and Figure \ref{fig:Slicing-6-7}-(B). \\
      \hline
      $A_t$ & Region covered by the faces in $\partial\Omega$ within distance $\delta_t$ of $p(x_t)$. See \eqref{def:At} and Figure \ref{fig:illustration-At}.\\
      \hline
      $\mathcal T_\fint$ & Set of two cells on each side of $\fint\in\Fhi$.\\
      \hline
      $A_g$ and $\Delta_g$ & Union and symmetric difference of $A_t$ and $A_{t'}$ for $t,t'\in\mathcal T_g$.\\
      \hline
      $\projFace(\fint)$ & Selected boundary face containing $p(x_\fint)$ in its closure. See \eqref{eq:def.projFace} and Figure \ref{fig:proj-g}.\\
      \hline
      $\projFace^\dagger(\fbou)$ & Set of internal faces \emph{approximately} above $\fbou$. See \eqref{eq:def.projFacedagger}.\\
      \hline
  \end{tabular}
  \caption{Sets of faces and regions}
  \label{Table:Notation_Sets}
\end{table}

\section*{Acknowledgments}
We thank Alberto F.~Mart\'in and Jordi Manyer for their support while implementing the numerical experiment using Gridap. 
This research was partially funded by the Australian Government through the Australian Research Council
(project numbers DP210103092 and DP220103160). We also acknowledge the funding of the European Union through the ERC Synergy scheme (NEMESIS, project number 101115663). Views and opinions expressed are however those of the authors only and do not necessarily reflect those of the European Union or the European Research Council Executive Agency. Neither the European Union nor the granting authority can be held responsible for them.

\printbibliography

\end{document}